\documentclass[a4paper]{article}
\usepackage{mathrsfs}
\usepackage{amsmath}
\usepackage{bbm}
\usepackage{amsfonts}
\usepackage{amssymb}
\usepackage{amscd}
\usepackage{amsthm}
\usepackage{scrtime}
\usepackage[mathlines]{lineno}
\usepackage[all,2cell]{xy}
\SelectTips{cm}{}
\theoremstyle{plain}

\newtheorem{theorem}{Theorem}

\newtheorem{thm}{Theorem}[subsection]
\newtheorem{cor}[thm]{Corollary}
\newtheorem{lem}[thm]{Lemma}
\newtheorem{prop}[thm]{Proposition}

\theoremstyle{definition}
\newtheorem{defn}[thm]{Definition}
\newtheorem{exm}[thm]{Example}

\theoremstyle{remark}
\newtheorem{rem}[thm]{Remark}

\numberwithin{equation}{section}

\newcommand\Gal{\operatorname{Gal}\nolimits}
\newcommand\Inv{\operatorname{Inv}\nolimits}

\renewcommand\dim{\operatorname{dim}}

\newcommand\Res{\operatorname{Res}\nolimits}
\newcommand\id{\operatorname{Id}\nolimits}

\newcommand\Hom{\operatorname{Hom}\nolimits}

\newcommand\Ker{\operatorname{Ker}\nolimits}

\newcommand\im{\operatorname{Im}\nolimits}
\newcommand\inn{\operatorname{Inn}\nolimits}

\newcommand\rad{\operatorname{Rad}\nolimits}

\newcommand\p{\mathbb{P}}
\newcommand\aut{\operatorname{Aut}\nolimits}
\newcommand\graut{\operatorname{GrAut}\nolimits}
\newcommand\stimes{{\scriptscriptstyle\times}}
\newcommand\splus{{\scriptscriptstyle +}}
\newcommand\sperp{{\scriptscriptstyle \perp}}
\newcommand\iomega{\operatorname{I\Omega}\nolimits}

\def\1{\mathbbm{1}}
\newcommand\f{\mathbbm{f}}
\newcommand\mbt{\mathbbm{t}}

\newcommand\C{\mathscr{C}}
\newcommand\G{\mathscr{G}}

\newcommand\cmu{{\check\mu}}

\begin{document}

\title{Automorphisms of path coalgebras and applications}
\author{Yu Ye}
\date{}
\maketitle
\footnote{Supported by National Natural Science Foundation of China (No. 10971206)}

\begin{abstract}
Our main purpose is to introduce the notion of trans-datum for quivers, and apply it to the
study of automorphism groups of path coalgebras and algebras. We observe that any homomorphism of path coalgebras is uniquely determined by a
 trans-datum, which is the basis of our work.
Under this correspondence, we show for any quiver $Q$
an isomorphism from $\aut(kQ^c)$ to $\Omega^*(Q)$, the group of invertible trans-data from $Q$ to
itself. We point out that the coradical filtration gives to a tower of normal subgroups of $\aut(kQ^c)$ with
all factor groups determined. Generalizing this fact, we establish a Galois-like theory for acyclic quivers,
which gives a bijection between large subcoalgebras of the path coalgebra and their Galois groups, relating
large subcoalgebras of a path coalgebra with certain subgroups of its automorphism group. The group
$\aut(kQ^c)$ is discussed by studying its certain subgroups, and the corresponding trans-data are given
explicitly. By the duality between reflexive coalgebras and algebras, we therefore obtain some structural
results of $\aut(\widehat{kQ^a})$ for a finite quiver $Q$, where $\widehat{kQ^a}$
is the complete path algebra. Moreover, we also apply these results to finite dimensional elementary algebras
and recover some classical results.
\end{abstract}

\setcounter{tocdepth}{2}

\tableofcontents

\parskip=0.15cm
\setcounter{section}{-1}
\section{Introduction}

Automorphism group of a mathematical object is, in roughly speaking, the symmetry of the object. A good knowledge about automorphisms of an object is essential in understanding its structure. To determine an automorphism group is always fundamental in classification problems. The present work aims to investigate automorphism groups of an important class of coalgebras and algebras, say the ones obtained from quivers.

Let $A$ be an associative algebra over a field $k$. In case that $A$ is finite dimensional, the automorphism group $\aut(A)$ is known to be an affine algebraic group with Lie algebra $\mathrm{Der}(A)$, the algebra of derivations of $A$; see \cite[$\S$1.2.3, ex.2]{ov}, for example. Some other important groups related are the outer automorphism group $\mathrm{Out}(A)=\aut(A)/\inn(A)$, the quotient group modulo the inner ones, and the Picard group $\mathrm{Pic}(A)$, the group of isoclasses of invertible $A$-$A$-bimodules.

Fr\"{o}hlich studied the Picard group of associative rings systematically. The Picard group is shown to be invariant
under Morita equivalence, and for any ring $R$ there exists a map from $\aut(R)$ to $\mathrm{Pic}(R)$ with kernel
 $\inn(R)$; see \cite{fr}. Moreover, Bolla \cite{bo} proved that when $R$ is basic semiperfect, then the map is also surjective, that is $\mathrm{Pic}(R)\cong\mathrm{Out}(R)$ in this case.

How the assumption $\aut(A)$ being reductive (respectively, semi-simple, toral, solvable, nilpotent)
effects the algebra structure of $A$ has been discussed by Pollack \cite{po}. It is worth mentioning that the solvability of the automorphism group is of
great importance in the classification of isolated hypersurface singularities,
see a series papers \cite{y1, y2} by Yau.

Usually to determine $\aut(A)$ (respectively, $\mathrm{Out}(A)$, $\mathrm{Pic}(A)$) for a
finite dimensional algebra $A$ is very difficult, even to find out all idempotent elements
of $A$ is not an easy task. Only a few scattered examples are known. A well studied example is the automorphism group of an incidence algebra beginning with Stanley \cite{st} and continued by Scharlau \cite{sch}, Baclawski \cite{ba}, and Coelho \cite{co}. Other examples are exterior algebras by Djokovi\'{c} \cite{dj}, square-free algebras by Anderson and D'Ambrosia \cite{aa}, and
monomial algebras by Guil-Asensio and Saor\'{i}n \cite{gs2}.

In a series papers \cite{gs1, gs2}, Guil-Asensio and Saor\'{i}n developed a strategy to compute the outer automorphism group and the
Picard group for finite dimensional algebras. As an application the Picard group of a split finite
dimensional algebra is calculated  in several special cases.

The present work is intended to be from a different point of view. We begin with the study of automorphisms of coalgebras and then apply to algebras. Notice that algebra and coalgebra are dual categorically. In philosophy, algebra is easy to handle sometimes, while in some other cases, to deal with coalgebras is easier. The key observation is that automorphisms of a path coalgebra are much easier to characterize. This is our starting point. We make an attempt to study the automorphism group (respectively, outer automorphism group, Picard group) of a path algebra via its dual path coalgebra. Recall that in path algebra case, the outer automorphism group and the Picard group are the same.

We remark that both incidence algebras and exterior algebras are elementary, which can be realized as quotient algebras of path algebras. It is possible to apply our general results on path algebras and coalgebras to these examples, although we will not go further in this direction in this paper.

Throughout, $k$ will be a fixed field and all
unadorned $\otimes$ will mean $\otimes_{k}$. It is worth mentioning that most results hold true in general situation, and quivers considered are arbitrary quivers unless otherwise stated. More precisely, we need the "acyclic" condition in Section 4 for the establishment of the Galois theory, and in some scattered applications we need certain finiteness condition.

The main results and the structure of the paper is outlined as follows.

In Section 1, we recall some basic notions and give a brief introduction to quiver techniques needed.

Section 2 deals with a characterization of path coalgebra homomorphisms.
We introduce the notion of trans-datum for quivers, see Definition \ref{def-trans-datum}. Let $Q$ and $Q'$ be quivers. We denote by $\Omega(Q,Q')$ the set of trans-data from $Q$ to $Q'$ and write for brevity $\Omega(Q) = \Omega(Q, Q)$. The following fundamental result is given in Theorem \ref{thm-monomial-extension}.

\begin{theorem} Let $Q, Q'$ be quivers, and $kQ^c$ and $kQ'^c$ the path coalgebras. Then we have mutually inverse maps $ \xymatrix @C=2.5pc {\operatorname{Coalg}(kQ^c, kQ'^c) \ar@<+0.5ex>[r]^(.6){ \mathbbm{t}} & \Omega(Q, Q')\ar@<+0.5ex>[l]^(.4){\mathbbm{f}} }$, here $\operatorname{Coalg}(kQ^c, kQ'^c)$ denotes the set of coalgebra homomorphisms from $kQ^c$ to $kQ'^c$.
\end{theorem}

Under the above correspondence, a composition map of trans-data, which is compatible with the one of coalgebra homomorphisms, is given explicitly, making $\Omega(Q)$ a monoid which is isomorphic to $\mathrm{Coalg}(kQ^c, kQ^c) $; see  Theorem \ref{thm-composion-map}.

The following results is obtained in Theorem \ref{thm-large-auto-ext} and \ref{thm-autogp-fd-alg-lifting}, which show us the advantage of using the notion of trans-datum, especially in the study of automorphisms of pointed coalgebras and elementary algebras.

\begin{theorem}
(1) Let $Q$ be any quiver and $C$ a large subcoalgebra of $kQ^c$. Then any automorphism of $C$ extends to an automorphism of $kQ^c$.

(2) Let $A$ be a finite dimensional elementary algebra and $Q$ its extension quiver. Then any automorphism of $A$ is induced from an automorphism of $\widehat{kQ^c}$.
\end{theorem}

The above theorem just says that the automorphism group of a finite dimensional elementary algebra is somehow controlled by the automorphism group of a complete path algebra. We expect similar result holds true for any finite dimensional algebra which satisfies certain separable condition, although we do not have an argument at moment.

We may compare the above theorem with Quebbemann's result \cite{qu}, which connects automorphisms of a tensor algebra and the ones of its quotient algebras. Let $A$ be a finite dimensional algebra with Jacobson radical $J$. Suppose that there is an algebra embedding $A/J\hookrightarrow A$ such that $A= A/J\oplus J$ as $A/J$-bimodules. Quebbemann showed that for any automorphism $\tau$ of $A$, there exists an automorphism $\tau'$ of the tensor algebra $T_{A/J}(J/J^2)$ and an algebra epimorphism $\phi\colon T_{A/J}(J/J^2)\longrightarrow A$, such that $\tau\circ\phi = \phi\circ\tau'$. Notice that this correspondence does not give a group homomorphism, for $\phi$ varies for different $\tau$.

Section 3 is devoted to the study of $\aut(kQ^c)$ for a quiver $Q$. The coradical filtration is shown to give a tower of normal subgroups of $\aut(kQ^c)$ with all factor groups characterized explicitly, particularly we calculate $\dim(\aut(kQ^c_{\le n}))$ in Proposition \ref{prop-dim-auto-coalg}, where $kQ^c_{\le n}$ is the $n$-truncated path coalgebra of $Q$. Consequently, finite dimensional truncated path coalgebras with solvable automorphism group are classified as follows; see Theorem \ref{thm-schurian-coalg} for detail.

\begin{theorem}
Let $Q$ be a finite quiver and $n\ge 2$ an integer. Then $\aut_0(kQ^c_{\le n})$ is solvable if and only if $Q$ is a Schurian quiver; $\aut(kQ^c_{\le n})$ is solvable if and only if $Q$ is a Schurian quiver and $\aut(Q)$ is resolvable.
 \end{theorem}

Compare also with Proposition \ref{prop-sol-autgp-trunalg} which concerns truncated path algebras with solvable automorphism groups. We mention that the solvability of the identity component of the automorphism group of a monomial algebra has been discussed in \cite{gs2}. Since truncated path algebras are obviously monomial and $\aut_0(kQ^c_{\le n})$ is the identity component of $\aut(kQ^c_{\le n})$ (see Remark \ref{rem-connectness-coalg}), thus the first part of Proposition \ref{prop-sol-autgp-trunalg} can also be deduced from  \cite[Corollary 2.22]{gs2}.

In Section 4 we develop a Galois-like theory for an acyclic quiver $Q$, connecting admissible subgroups of $\aut(kQ^c)$ and large subcoalgebras of $kQ^c$. The following result, which we call the fundamental theorem for Galois extensions, is given in Theorem \ref{thm-galois-cor-pt-coal}.

\begin{theorem} Let $Q$ be an acyclic quiver. Let $D$ be a large subcoalgebra of $kQ^c$ and $E/D$ a Galois extension. Then there is a bijection between $\C(D,E)$, the set of intermediate coalgebras, and $\G(D,E)$, the set of admissible subgroups of $\Gal(E/D)$; Moreover, for any intermediate coalgebra $D\subseteq M\subseteq E$, $M/D$ is a Galois extension if and only if $\Gal(E/M)\lhd\Gal(E/D)$, and in this case, $\Gal(E/D)/\Gal(E/M)\cong \Gal(M/D).$
\end{theorem}

We refer to Section 4 for an explanation of notations. Some applications of the fundamental theorem are also given there.

Section 5 concerns with the study of automorphisms of path coalgebra and complete path algebras. The
idea is to consider for a quiver $Q$ the following natural subgroups of $\Omega^*(Q)$, say $\iomega^*(Q)$, $\iomega^*_0(Q)$, $\iomega^*_\circ(Q)$ and
$\Omega_\bullet^*(Q)$. We show that $\iomega^*(Q)\lhd\Omega^*(Q)$, $\iomega^*_\circ(Q)\lhd\Omega^*(Q)$ and $\iomega^*(Q) = \iomega^*_\circ(Q) \rtimes \iomega^*_0(Q)$, see Proposition \ref{prop-inn-coalg-is-normal} for detail.

We set $\inn(kQ^c)=\f(\iomega^*(Q))$, $\inn_0(kQ^c)=\f(\iomega^*_0(Q))$, $\inn_\circ(kQ^c)=\f(\iomega^*_\circ(Q))$ and
$\aut_\bullet(kQ^c)=\f(\Omega_\bullet^*(Q))$.
We mention that for $\mu\in\iomega^*(Q)$, a useful formula for $\f_\mu$ is obtained in Proposition \ref{prop-inn-autocoalg-formular}. Consequently each $\sigma\in \inn(kQ^c)$ acts invariantly on any large subcoalgebra of $kQ^c$; see Corollary \ref{cor-inn-preserving}.

Let $D$ be a finitary pointed
coalgebra, here finitary means that $D_{(1)}$ is finite dimensional. $D$ is realized as a large
subcoalgebra of $kQ^c$ for some finite quiver $Q$.
Due to Taft we have a group isomorphism $\aut(D)\cong\aut(D^*)$, particularly
$\aut(\widehat{kQ^a})\cong \aut(kQ^c)$ and hence $\aut(\widehat{kQ^a})\cong \Omega^*(Q)$ for any
finite quiver $Q$; see Section 5.1.1.
We rely heavily on this basic fact, which enables to apply the technique developed for path coalgebras.

Denote by $\inn(\widehat{kQ^a})$ the inner automorphism group of $\widehat{kQ^a}$. Consider also
 $\inn_\circ(\widehat{kQ^a})$, the normal subgroup of $\inn(\widehat{kQ^a})$ induced by elements in
 $1+ \rad(\widehat{kQ^a})$, and the subgroup $\inn_0(\widehat{kQ^a})$
of inner automorphisms induced by elements in $(kQ_0)^\stimes$. Another subgroup taken into account is
$\aut_\bullet(\widehat{kQ^a})$, the subgroup of automorphisms fixing $kQ_0$. Similar
subgroups are defined for quotient algebras of $\widehat{kQ^a}$. Notice that the subgroups
$\inn_\circ(kQ^a)$ and $\aut_\bullet(kQ^a)$ have been studied by many authors, see for example \cite[Section 1]{po} and \cite{gs1}.

By Proposition \ref{prop-cor-inn-sbgps} and \ref{prop-cor-bullet-sbgp}, $\inn(\widehat{kQ^a})$, $\inn_\circ(\widehat{kQ^a})$, $\inn_0(\widehat{kQ^a})$ and $\aut_\bullet(\widehat{kQ^a})$ correspond to
$\inn(kQ^c)$, $\inn_\circ(kQ^c)$, $\inn_0(kQ^c)$ and $\aut_\bullet(kQ^c)$ respectively. In this sense, elements in $\inn(kQ^c)$ are also called inner automorphisms of $kQ^c$.
Now we have the following characterization, which is given in Theorem \ref{thm-antiiso-dual-alg}.

\begin{theorem} Let $Q$ be a finite quiver, $A= \widehat{kQ^a}$ and $B=A/I$, where $I\subseteq \widehat{kQ^a}_{\ge2}$ is an ideal of $\widehat{kQ^a}$. Then $\aut_0(B) = \inn_\circ(B)\aut_\bullet(B)$ and $\inn(B)\cap \aut_\bullet(B) = \chi(Z^\stimes_{B}(B_0))$.
If moreover, $Z^\stimes_{B}(B_0) = B_0^\stimes$, then $\aut_0(B) = \inn_\circ(B)\rtimes\aut_\bullet(B)$.
\end{theorem}

Apply to finite dimensional case, we reobtain some classical results. For instance, we calculate the dimension
of $\mathrm{Out}(kQ^a)$ for acyclic quivers, and consequently show that $\mathrm{Out}(kQ^a)$ is finite
if and only if the quiver is a tree; see Corollary \ref{cor-dim-outer-gp} and
\ref{cor-autogp-eq-inngp} for more detail.

Examples are provided in Section 6. We discuss the quivers of directed $A_n$-type, $n$-Kronecker quivers and $n$-subspace quivers there. It is known that for a non-acyclic quiver, the automorphism group of the path algebra and the one of the path coalgebra are not isomorphic under the natural correspondence. We also illustrate this fact with a typical example.

\section{Qivers, path algebras and coalgebras}

We will give a brief introduction to quiver techniques in this section. For more details we refer to \cite{ars} and \cite{ri}.

\subsection{Quivers}
Recall that a \textbf{quiver} $Q=(Q_0, Q_1, s, t\colon Q_1\longrightarrow Q_0)$ is a directed graph, where $Q_0$ is the set of vertices, $Q_1$ the set of edges (usually called arrows), and $s, t\colon Q_1\longrightarrow Q_0$ are two maps assigning for each arrow its starting and terminating vertex respectively.
A quiver $Q$ is said to be \textbf{finite} if $Q_0$ and $Q_1$ are both finite, and \textbf{acyclic} if $Q$ contains no oriented cycles as subquiver.

By a \textbf{nontrivial path} in $Q$ we mean a sequence of arrows $p=\alpha_1\alpha_2\cdots\alpha_n$ with $s(\alpha_i)=t(\alpha_{i-1})$ for all $2\le i\le n$, $n$ is called the length of $p$, denoted by $l(p) = n$,  pictorially
\[\bullet\xrightarrow{\alpha_1}\bullet\xrightarrow{\alpha_2}\bullet\xrightarrow{\alpha_3}\cdots\xrightarrow{\alpha_n}\bullet.\] We use $s(p)=s(\alpha_1)$ and
$t(p)=t(\alpha_n)$ to denote the starting and terminating vertex of $p$ respectively and say that $p$ is a path from $s(p)$ to $t(p)$. For each vertex $i\in Q_0$,
we use $e_i$ to denote the \textbf{trivial path}, i.e., a path of length 0, from vertex $i$ to itself.

By abuse of notations, we also use $Q$ to denote the set of all paths in $Q$ and usually $P$ the set of all nontrivial paths in $Q$. For $i, j\in Q_0$, $Q_{i,j}$ (respectively, $P_{i,j}$) denotes the set of paths (respectively, nontrivial paths) from $i$ to $j$.

As usual $kQ$ denotes the $k$-space spanned by all paths in $Q$. Clearly, $kQ=\bigoplus_{n\ge 0}kQ_n$ is a positively graded space, here $Q_n$ is the set of paths in $Q$ which are of length $n$. Similarly set $(Q_n)_{i,j}\subseteq Q_n$ to be the subset consisting of those paths from $i$ to $j$ for any $i,j\subseteq Q_0$. Clearly $kQ$ is finite dimensional if and only if $Q$ is a finite acyclic quiver.

The \textbf{path algebra} of the quiver $Q$ has the underlying vector space $kQ$ and multiplication given by concatenation of paths in an obvious way, that is, the product of two paths $x$ and $y$ is the path
$xy$ if $t(x)=s(y)$ and 0 otherwise. We denote the path algebra by $kQ^a$.  By definition $e_i$ is an idempotent for each $i\in Q_0$, and $kQ^a$ has an identity element if and only if $Q_0$ is a finite set, and in this case, $1_{kQ^a}=\sum_{i\in Q_0}e_i$. The path algebra is a positively graded algebra with respect to the length grading.

The \textbf{path coalgebra} of $Q$, denoted by $kQ^c$, is in some sense the dual of the path algebra. As a vector space $kQ^c=kQ$. The coproduct $\Delta$ is given by splitting the path at all possible position, to be precise, $\Delta(e_i)=e_i\otimes e_i$ for each $i\in Q_0$ and
\[\Delta(p)=s(p)\otimes p + \sum_{1\le i\le n-1}\alpha_1\alpha_2\cdots\alpha_i\otimes \alpha_{i+1}\alpha_{i+2}\cdots\alpha_n+ p\otimes t(p)\] for each nontrivial path $p=\alpha_1\alpha_2\cdots\alpha_n$.
The counit is given by $\varepsilon(e_i)=1$ for each vertex $i$ and $\varepsilon(p)=0$ for each nontrivial path $p$. Again the path coalgebra is a positively graded coalgebra with respect to the length grading.

\begin{rem} Recall that there is a more general construction by using tensor product and cotensor product. Let $C$ be a coalgebra and $M^C$ and $^CN$ be $C$-comodules with the coaction given by $\phi\colon M\to M\otimes C$ and $\psi\colon N\to C\otimes N$. The cotensor product $M \Box_C N$ is defined to be the kernel of \[\phi\otimes\id_N-\id_M\otimes\psi\colon M\otimes N\longrightarrow M\otimes C\otimes N,\] we refer to \cite{em} for more detail about cotensor products.

Note that for a quiver $Q$, the path algebra $kQ^a$ is isomorphic to $\operatorname{T}_{kQ_0}(kQ_1)$, the tensor algebra of the $kQ_0$-bimodule $kQ_1$ over $kQ_0$; and the path coalgebra $kQ^c$ is isomorphic to the cotensor coalgebra $\operatorname{Cot}_{kQ_0}(kQ_1)$ of the $kQ_0$-bicomodule $kQ_1$ over $kQ_0$. We therefore identify the $k$-space $kQ_n$ with the tensor space $\operatorname{T}^n_{kQ_0}(kQ_1)$ as well as the cotensor space $\operatorname{Cot}_{kQ_0}^n(kQ_1)$.

\end{rem}

\subsection{Taft-Wilson Theorem and coalgebra filtration}
Let $C$ be a coalgebra. A nonzero element $g$ in $C$ is called a \textbf{group-like element} if $\Delta(g)=g\otimes g$. The set of group-like elements in $C$ is denoted by $G(C)$. Let $g, h$ be two group-likes, a $(g,h)$-\textbf{primitive element} is by definition an element $\alpha$ such that $\Delta(\alpha)=g\otimes \alpha + \alpha\otimes h$. We denote the set of $(g,h)$-primitive elements by $\p_{g,h}$, which is easy shown to be a $k$-linear space. For consistency of notations, we set $\p_{g,h} = 0$ if either $g = 0$ or $ h = 0$. In case $C=kQ^c$ is the path coalgebra of some quiver $Q$, then we simply write $\p_{s,t}=\p_{e_s, e_t}$ for any $s, t\in Q_0$.

 We have the following easy characterization for path coalgebras, which is known as the Taft-Wilson Theorem, see also \cite[Theorem 5.4.1]{dnr} for a proof.

\begin{lem}{\rm(\cite{tw})}\ Let $Q$ be an arbitrary quiver and $kQ^c$ the path coalgebra. Then
\begin{enumerate}
\item[(1)] $G(kQ^c) = \{e_i| i\in Q_0\} $;
\item[(2)] $\p_{s, t} = k(e_s-e_t)\bigoplus k(Q_1)_{s,t}$ for each pair of vertices $s, t\in Q_0$; in particular, $\p_{s,s} = k(Q_1)_{s,s}$ for each vertex $s\in Q_0$.
\end{enumerate}
\end{lem}

Let $C$ be an arbitrary coalgebra. There exists a unique filtration of subcoalgebras of $C$, say the \textbf{coradical filtration}
\[C_{(0)}\subseteq C_{(1)}\subseteq C_{(2)}\subseteq C_{(3)}\subseteq
\cdots, \]
 such that $C_{(0)}$ is the coradical of $C$, that is the sum of all simple subcoalgebras of $C$, and $\Delta(C_{(n)})\subseteq \sum\limits_{0\le i\le n} C_{(i)}\otimes C_{(n-i)}$.
If $C=kQ^c$, the path coalgebra of a quiver $Q$, then $kQ^c_{(n)}= kQ_{\le n}=\sum\limits_{0\le i\le n}kQ_i$. The following useful lemma is due to Heyneman and Radford (\cite{hr}); see also \cite[p.65]{dnr} for a proof.

\begin{lem}\label{corad}Let $C$ and $D$ be coalgebras and $f\colon C\longrightarrow D$ a coalgebra morphism. Then $f(C_{(n)})\subseteq D_{(n)}$, and $f$ is injective if and only if $f|_{C_{(1)}}$ is injective. In particular, $f(G(C))\subseteq G(D)$ and $f(\p_{g,h})\subseteq \p_{f(g),f(h)}$.
\end{lem}

\subsection{Complete path algebra, the dual of path coalgebra}
 In this subsection $Q$ is assumed to be a finite quiver. Note that the path algebra $kQ^a$ and path coalgebra $kQ^c$ have the same underlying space $kQ$, thus they share the same basis $Q$ consisting of all paths. To avoid confusion, we use $\overline{Q}=\{\overline{p}|p\in Q\}$ to denote the corresponding basis of $kQ^a$. We may introduce a non-degenerate pairing $(-,-)\colon kQ^a \times kQ^c \longrightarrow k$ by setting $(\overline{p}, q)=\delta_{p,q}$, for any $p, q\in Q$. Thus for each $n\ge 0$, the homogeneous components $kQ^a_n$ and $kQ^c_n$ are dual to each other as vector spaces via $(-,-)$. Moreover, if $Q$ is acyclic, then $kQ^a$ and $kQ^c$ are dual to each other.

The notion of \textbf{complete path algebra} $\widehat{kQ^a}$ for a quiver $Q$ is introduced by Derksen, Weyman and Zelevinsky to study the mutation of quivers with potentials; see \cite[Definition 2.2]{dwz}. As $k$-spaces, $\widehat{kQ^a} = \prod_{p\in Q}k =\{ (a_p)_{p\in Q}, a_p\in k \}$, and the multiplication in $ \widehat{kQ^a}$ is given by
\[(a_p)_{p\in Q}\cdot (b_p)_{p\in Q} = (c_p)_{p\in Q},\ c_p = \sum_{p=p_1p_2\atop p_1, p_2\in Q, }a_{p_1}b_{p_2}. \]
Here by $p= p_1p_2$ we mean a split of the path $p$ such that the concatenation of $p_1$ and $p_2$ is $p$. An element $(a_p)_{p\in Q}\in \widehat{kQ^a}$ can also be written as $\sum_{p\in Q} a_p\overline p$, an infinite linear combinations of paths in $Q$. In this sense, $kQ^a$ is a subalgebra of $\widehat{kQ^a}$ and $kQ^a = \widehat{kQ^a}$ if and only if $Q$ is a finite acyclic quiver.

Set $J=kQ^a_{\ge 1}$ be the Graded Jacobson ideal of $kQ^a$. Then we have a filtration of ideals $kQ^a\supseteq J\supseteq J^2\supseteq J^3\supseteq\cdots$ of $kQ^a$, where each
$J^n= kQ^a_{\ge n}$. One shows that  $\widehat{kQ^a}$ is nothing but the completion of $kQ^a$ at the ideal $J$, i.e., the inverse limit $\varprojlim kQ^a/J^n$.

There is a uniquely determined non-degenerate pairing $(-,-)\colon \widehat{kQ^a} \times kQ^c \longrightarrow k$, extending the one between $kQ^a$ and $kQ^c$ introduced above. Under this pairing, $\widehat{kQ^a}$ is isomorphic to $(kQ^c)^*$, the dual vector space of $kQ^c$. The following basic facts can be found in \cite[Chapter5]{dnr}, and we list them without a proof for later use.

\begin{lem} Let $(C,\Delta, \varepsilon)$ be a coalgebra and $C^*$ its dual vector space.

(1) $C^*$ is an associative algebra with the identity element $\varepsilon$ and the multiplication given by the convolution $*$, that is, $(f*g)(c) := \sum f(c_{(1)})g(c_{(2)})$ for any $f, g\in C^*$ and $c\in C$, here we use the Sweedler's notation $\Delta(c)=\sum c_{(1)}\otimes c_{(2)}$.

(2) Let $D$ be a subcoalgebra of $C$. Then $D^*\cong C^*/D^\sperp$ as an associative algebra, where $D^\sperp = \{f\in C^*\mid f(d) = 0, \forall d\in D\}$ is an ideal of $C^*$.

(3) Let $C_{(0)}\subseteq C_{(1)}\subseteq C_{(2)}\subseteq \cdots $ be the coradical filtration of $C$. Then $\rad^n(C^*)= C_{(n-1)}^\sperp$ for each $n\ge 1$, here $\rad(C^*)$ is the Jacobson radical of $C^*$ and $\rad^n(C^*) = \rad(C^*)\rad^{n-1}(C^*)$ for each $n\ge2$.
\end{lem}

Turning back to quiver case, one has the following results.
\begin{lem} Let $Q$ be a finite quiver. Then

(1)\quad $\widehat{kQ^a}\cong (kQ^c)^*$ as associative algebras, and $kQ^a\cong (kQ^c)^*$ if $Q$ is acyclic;

(2)\quad $\rad^n(\widehat{kQ^a}) = \{\sum_{p\in Q, l(p)\ge n} a_p \overline p\mid a_p\in k\}$;

(3)\quad A subcoalgebra $D\subseteq kQ^c$ is large if and only if $D^\sperp\subseteq\rad(\widehat{kQ^a})$.
\end{lem}

Recall that a subcoalgebra $D$ of a path coalgebra $kQ^c$ is said to be \textbf{large} if $D\supseteq kQ_{\le1}$. A classical result by Chin and Montgomery says that any pointed coalgebra $D$ can be realized as a large subalgebra of the path coalgebra $kQ^c$ of some quiver $Q$, and such $Q$ is uniquely determined by $D$ and usually called the extension quiver of $D$; see \cite[Theorem 4.3]{cm}.

An ideal $I$ of $kQ^a$ is called an \textbf{admissible ideal} if $kQ^a_{\ge n}\subseteq I\subseteq kQ^a_{\ge2}$ for some $n\ge2$. Note that $kQ^a/I$ is a finite dimensional elementary algebra if $Q$ is a finite quiver and $I$ is admissible, here elementary means that each simple module is one dimensional. In this case, $\widehat{kQ^a}/\widehat{I}\cong kQ^a/I$, here $\widehat{I} = I\widehat{kQ^a}$ is an ideal of $\widehat{kQ^a}$ satisfying $\rad^n(\widehat{kQ^a}) \subseteq \widehat{I}\subseteq \rad^2(\widehat{kQ^a})$. We also call an ideal of $\widehat{kQ^a}$ with this property an admissible ideal.

Conversely, a famous result by Gabirel says that any finite dimensional elementary algebra $A$ has the form $kQ^a/I$, where $Q$ is a finite quiver uniquely determined by $A$ and $I$ an admissible ideal of $kQ^a$; see \cite[Theorem 3.6.6]{dk}. $Q$ is usually called the extension quiver of $A$. Clearly, the dual coalgebra $C = A^*$ is a large subcoalgebra of $kQ^c$.

\subsection{Augmented quivers}
We may associate each quiver an auxiliary quiver, which plays an important role in our characterization of path coalgebra homomorphisms.

\begin{defn} Let $Q$ be a given quiver. The \textbf{augmented quiver} of $Q$, denoted by  $Q^{\tau}$, is a quiver obtained from $Q$ by adding some new arrows, precisely $Q^{\tau}_0=Q_0$, and $Q^{\tau}_1=Q_1\cup\{e_{s,t}\mid s,t\in
Q_0, s\ne t\}$ with $s(e_{s,t})=s$ and $t(e_{s,t})=t$ for each $e_{s,t}$.
\end{defn}

\begin{rem}
For a given quiver $Q$, we can identify $\p_{s,t}$ with $k(Q^{\tau}_1)_{s,t}$ for any $s, t\in Q_0$, here $(Q^{\tau}_1)_{s,t}$ denotes the set of arrows in the augmented quiver $Q^{\tau}$ from $i$ to $j$.
\end{rem}

\begin{exm} Let $Q$ be the quiver $\xymatrix{
 {\bullet}1 \ar @{->}[r]^{\alpha}
& {\bullet}2 \ar @{->}[r]^{\beta}
& {\bullet}3}$, then the augmented quiver $Q^{\tau}$ is given by
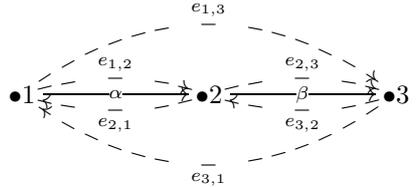
\[\xymatrix{
 {\bullet}1 \ar @{->}[rr]|{\alpha} \ar @{-->}@/^.5pc/[rr]^{e_{1,2}}\ar @{-->}@/^2.2pc/[rrrr]^{e_{1,3}} &
& {\bullet}2 \ar @{-->}@/^.5pc/[ll]^{e_{2,1}} \ar @{->}[rr]|{\beta} \ar @{-->}@/^.5pc/[rr]^{e_{2,3}} &
& {\bullet}3 \ar @{-->}@/^.5pc/[ll]^{e_{3,2}} \ar @{-->}@/^2.2pc/[llll]^{e_{3,1}}},
\] here we use dashed arrows to denote the new added ones.

\end{exm}

As before let $Q$ denote the set of all paths in $Q$, $P\subset Q$ the set of
all nontrivial paths in $Q$ and $Q^{\tau}$ the set of all paths in $Q^{\tau}$. $Q$ and $P$ are identified as
subsets of $Q^{\tau}$ in an obvious way. We also use $E$ to denote the set of all nontrivial paths
in $Q^{\tau}$ which are compositions of new arrows.

 We define a linear map $F: k(Q^{\tau})^c\longrightarrow kQ^c$ by
setting {\begin{eqnarray*}
&F(p)=p, &\forall \ p\in Q;\\
&F(e)=e_{i_1}-e_{i_2},&\forall\
e= e_{i_1,i_2}e_{i_2,i_3}\cdots e_{i_n,i_{n+1}}\in E;\\
&F(pe)=p, &\forall\ p\in P, e\in E;\\
&F(e_{s,s(p)}p)=-p, &\forall\ p\in P,
 s\neq s(p)
\in Q_0;\\
&F(e_{s,s(p)}pe)=-p, &\forall\ p\in
P,
 s\neq s(p)\in Q_0, e\in E;\\
&F(p)=0, &otherwise.
\end{eqnarray*}}

To characterize homomorphisms of path coalgebras we need the following key lemma. The proof is given by routine check and we omit it here.

\begin{lem}\label{keylemma} Let $Q$ and $F$ be as above. Then $F$ is a homomorphism of coalgebras such that $F|_{kQ^c}=\id$ and
$F(e_{s,t})=e_s-e_t$ for any $s$, $t\in Q_0$.
\end{lem}

\begin{rem}  We emphasize that a coalgebra homomorphism from $ k(Q^{\tau})^c$ to $kQ^c$ satisfying the conditions in Lemma \ref{keylemma} is not unique in general.
\end{rem}

\section{Trans-data and path coalgebra morphisms }

We introduce the notion of trans-datum for quivers in this section. As we will show, homomorphisms
of path coalgebras correspond to trans-data bijectively. Under this correspondence,
trans-data can also be composed and the composition map is given explicitly.

\subsection{Trans-data}
\begin{defn} \label{def-trans-datum} Let $Q, Q'$ be quivers. A \textbf{trans-datum} $\mu$ from $Q$ to $Q'$ is by definition a pair $(\mu_0, \mu_\splus)$, where $\mu_0\colon kQ_0^c\longrightarrow Q'^c_0$ is a coalgebra homomorphism and $\mu_\splus =(\mu_p)_{p\in P}\in \prod_{p\in P}\p_{\mu_0(e_{s(p)}),\mu_0(e_{t(p)})}$ a family of primitive elements in $kQ'^c$ indexed by $P$. Trans-data are usually denoted by $\lambda, \mu, \nu, \cdots$. The set of all trans-data from $Q$ to $Q'$ is denoted by $\Omega(Q,Q')$ and we write $\Omega(Q) = \Omega(Q,Q)$ for brevity.
\end{defn}

Let $\mu = (\mu_0, \mu_\splus)\in \Omega(Q, Q')$.
Write $\mu_\splus = \mu^0+\mu^1$ in the sense that $\mu^0$ and $\mu^1$ are both families of primitives indexed by $P$ with $\mu^0_p\in kQ'_0$, $\mu^1_p\in kQ'_1$ and $\mu_p=\mu^0_p+\mu^1_p$ for any $p\in P$. Clearly $\mu^0$ and $\mu^1$ are uniquely determined by $\mu_\splus$ and hence $\mu$ can also be written as a triple $(\mu_0, \mu^0, \mu^1)$. Moreover, since $\mu^0_p$ is a multiple of ${\mu_0(e_{s(p)})-\mu_0(e_{t(p)})}$, we denote by $\cmu_p$ the coefficient in the expression $\mu^0_p = \cmu_p({\mu_0(e_{s(p)})-\mu_0(e_{t(p)})})$. Clearly we may set $\cmu=0$ whenever $s(p)=t(p)$. We also view $\mu^0$ and $\mu^1$ as linear maps from $kQ_{\ge 1}$ to $kQ'_0$ and $kQ'_1$ in an obvious way. Now we have the following $k$-linear maps associated to a trans-datum $\mu$:
\[\mu_0\colon kQ\longrightarrow kQ'_0,\ \mu_0(p) = 0 \ \forall p\in P;\]
\[\mu^0\colon kQ\longrightarrow kQ'_0, \ \mu^0(e_i)=0 \ \forall i\in Q_0,\ \mu^0(p)=\mu^0_p \ \forall p\in P;\]
\[\cmu\colon kQ\longrightarrow k, \ \cmu(e_i)=0\ \forall i\in Q_0, \ \cmu(p) = \cmu_p\ \forall p\in P;\]
\[\mu^1\colon kQ\longrightarrow kQ'_1, \ \mu^1(e_i)=0 \ \forall i\in Q_0,\ \mu^1(p)=\mu^1_p\ \forall p\in P;\]
\[\mu_\splus\colon kQ\longrightarrow kQ'_0\oplus kQ'_1,\ \mu_\splus = \mu^0 +\mu^1.\]

\begin{prop}\label{prop-coalgmap-to-datum}
Let $Q, Q'$ be any quivers and $Q^{\tau}, Q'^\tau$ the augmented quivers. Denote by $P$ the set of all nontrivial paths in $Q$. Let $\mu = (\mu_0, \mu_\splus)\in \Omega(Q, Q')$. We extend $\mu_0$ to a linear map $f\colon kQ^c\longrightarrow kQ'^c$ by setting
$ f|_{kQ_0} = \mu_0 $ and
\begin{equation}
  \label{endoformulae} f(p)=\sum_{p=p_1p_2\cdots p_r\atop p_i\in P, i=1,\cdots, r}F(\mu_{p_1}\Box \mu_{p_2}\cdots\Box \mu_{p_r})
\end{equation}
for any $p\in P$, here we view the path coalgebra as a cotensor coalgebra and $\Box = \Box_{kQ'_0}$ is the cotensor product. Then $f$ is a homomorphism of coalgebras. Moreover, $f$ is injective if and only $\mu_0$ and $\mu^1|_{kQ_1}$ are both injective.
\end{prop}
\begin{proof} It suffices to prove that $f$ is a coalgebra homomorphism, and the rest part follows from Lemma \ref{corad} directly. We will show that for any $p\in P$, $(f\otimes f)(\Delta(p)) = \Delta(f(p))$.  By direct calculation,

\begin{align*}
 RHS=& \sum_{p=p_1p_2\cdots p_r\atop {p_i\in P}}\Delta(F(\mu_{p_1}\Box\cdots\Box \mu_{p_r}))\\
    =& \sum_{p=p_1p_2\cdots p_r\atop {p_i\in P}}(F\otimes F)(\Delta(\mu_{p_1}\Box\cdots\Box \mu_{p_r}))\\
    =& \sum_{p=p_1p_2\cdots p_r\atop {p_i\in P}} \left( (f(e_{s(p)})\otimes F(\mu_{p_1}\Box\cdots\Box \mu_{p_r})\right.\\
     &+  \sum_{1\le i \le r-1} F(\mu_{p_1}\Box\cdots\Box \mu_{p_i})\otimes F(\mu_{p_{i+1}}\Box\cdots\Box \mu_{p_r})  \\
     &+  \left. (\mu_{p_1}\Box\cdots\Box \mu_{p_r})\otimes f(e_{t(p)}) \right)\\
    =& f(e_{s(p)})\otimes \sum_{p=p_1p_2\cdots p_r\atop {p_i\in P}}  F(\mu_{p_1}\Box\cdots\Box \mu_{p_r})\\
     & + \sum_{p=p_1p_2\cdots p_r\atop {p_i\in P}} \sum_{1\le i \le r-1} F(\mu_{p_1}\Box\cdots\Box \mu_{p_i})\otimes F(\mu_{p_{i+1}}\Box\cdots\Box \mu_{p_r})  \\
     & +\sum_{p=p_1p_2\cdots p_r\atop {p_i\in P}} (\mu_{p_1}\Box\cdots\Box \mu_{p_r})\otimes f(e_{t(p)}) \\
    =& f(e_{s(p)})\otimes f(p) \\
     & +\sum_{p=p_1p_2\cdots p_r\atop {p_i\in P}}\sum_{1\le i \le r-1} F(\mu_{p_1}\Box\cdots\Box \mu_{p_i})\otimes F(\mu_{p_{i+1}}\Box\cdots\Box \mu_{p_r})\\
     & + f(p)\otimes f(e_{t(p)}).
\end{align*}
and
\begin{align*}
 LHS=& (f\otimes f)(e_{s(p)}\otimes p + \sum_{p=p_1p_2\atop p_i\in P}p_1\otimes p_2 + p\otimes e_{t(p)}) \\
    =& f(e_{s(p)})\otimes f(p) + \sum_{p=p_1p_2\atop p_i\in P}f(p_1)\otimes f(p_2) + f(p)\otimes f(e_{t(p)}),
\end{align*}
Now by comparing different ways to write a path as a composition of subpaths, we have
\begin{align*} &\sum_{p=p_1p_2\atop p_i\in P}f(p_1)\otimes f(p_2)\\
= & \sum_{p=p_1p_2\atop p_i\in P}\sum_{p_1 = x_1x_2\cdots x_u\atop {p_2 = y_1y_2\cdots y_v\atop x_i, y_j\in P}}
 F(\mu_{x_1}\Box\cdots\Box \mu_{x_u})\otimes F(\mu_{y_1}\Box\cdots\Box \mu_{y_v})\\
= & \sum_{p=p_1p_2\cdots p_r\atop {p_i\in P}}\sum_{1\le i \le r-1} F(\mu_{p_1}\Box\cdots\Box \mu_{p_i})\otimes F(\mu_{p_{i+1}}\Box\cdots\Box \mu_{p_r}),
\end{align*}
and it follows that LHS=RHS, which completes the proof.
\end{proof}

The formula \ref{endoformulae} can also be rewritten as follows, which will be more practical in some special cases.

\begin{prop}\label{endoformulae2} Let $Q, Q'$ be quivers and $\mu\in\Omega(Q, Q')$. Then the coalgebra morphism $f$ obtained from
$\mu$ is given by
\begin{align*}\f_\mu(x)&= \mu_0(x) + \sum_{r\ge 1} F(\mu_\splus(x_{(1)})\Box\cdots \Box \mu_\splus(x_{(r)}))
\end{align*}
for any $x\in kQ^c_{\ge1}$, again we use the Sweedler's notation $\Delta^{n}(x)=\sum x_{(1)}\otimes \cdots \otimes x_{(n+1)}$.
\end{prop}
\begin{proof} We need only to prove the case $x=p$ for some nontrivial path $p$. By definition
\[\Delta^{r-1}(p) =\sum\limits_{p=q_1q_2\cdots q_r \atop q_i\in Q}q_1\otimes q_2\otimes\cdots\otimes q_r,\]
where $q_i$'s are allowed be trivia paths.
Note that $\mu_\splus(q_i) = 0$ if $q_i$ is a trivial path. Therefore $F(\mu_\splus(q_1)\Box \mu_\splus(q_2)\Box\cdots\Box\mu_\splus(q_r)) = 0$ whenever one of $q_i$'s is trivial.
It follows that
\begin{align*}& \mu_0(p) + \sum_{p=q_1q_2\cdots q_r\atop q_i\in Q}F(\mu_\splus({q_1})\Box \mu_\splus({q_2})\cdots\Box \mu_\splus({q_r}))\\
 = &\sum_{p=p_1p_2\cdots p_r\atop p_i\in P}F(\mu_\splus({p_1})\Box \mu_\splus({p_2})\cdots\Box \mu_\splus({p_r}))
\\ = & \sum_{p=p_1p_2\cdots p_r\atop p_i\in P}F(\mu_{p_1}\Box \mu_{p_2}\cdots\Box \mu_{p_r}) = \f_\mu(p).
\end{align*}
The assertion follows from the linearity of both sides in the equality.
\end{proof}

\subsection{Trans-data $=$ Path coalgebra homomorphisms}
Next we show that any coalgebra homomorphism of path coalgebras is obtained from a trans-datum in a way as given in Proposition \ref{prop-coalgmap-to-datum}.  Recall that a subcoalgebra $D$ of $kQ^c$ is said to be \textbf{monomial} if $D$ has a basis consisting of paths in $Q$. Note that if $p=\alpha_1\alpha_2\cdots\alpha_r\in D$, then $\alpha_i\alpha_{i+1}\cdots\alpha_j\in D$ for any $i\le j$, especially $e_{s(p)}, e_{t(p)}\in D$.

\begin{thm}\label{thm-monomial-extension} Let $Q, Q'$ be given quivers, $D\subseteq kQ^c$ a monomial subcoalgebra with a basis $X$ consisting of paths in $Q$ and $f\colon D\longrightarrow kQ'^c$ a coalgebra homomorphism. Then we can assign each nontrivial path $p\in X$ a primitive element $\mu_{p}\in \p_{f(e_{s(p)}),f(e_{t(p)})}$ in a unique way,  such that $(\ref{endoformulae})$ holds for any $p\in P\cap X $.
In particular, we have mutually inverse bijections  \[ \xymatrix @C=2.5pc {\operatorname{Coalg}(kQ^c, kQ'^c) \ar@<+0.5ex>[r]^(.6){ \mathbbm{t}} & \Omega(Q, Q')\ar@<+0.5ex>[l]^(.4){\mathbbm{f}} }, \] here $\operatorname{Coalg}(kQ^c, kQ'^c)$ is the set of coalgebra homomorphisms,  $\mathbbm{t}$ is given as above and $\mathbbm{f}$ given as in Proposition $\ref{prop-coalgmap-to-datum}$.
\end{thm}

\begin{proof} Set $X_n=\{p\in X\mid l(p)=n\}$ for each $n\ge 1$, here $l(p)$ denotes the length of $p$. We use induction on $l(p)$ to prove the existence and uniqueness of $\mu_p$.

For $p\in X_1$, we set $\mu_p=f(p)$, and $p\in \p_{s(p),t(p)}$ implies that $f(p)\in \p_{f(e_{s(p)}),f(e_{t(p)})}$. In this case, the formula (\ref{endoformulae}) holds automatically and it also determines $\mu_p$ uniquely.

Now assume that for each $p\in X$ with $1\le l(p)\le n-1$, we have assigned a primitive element $\mu_{p}\in \p_{f(e_{s(p)}),f(e_{t(p)})}$ such that (\ref{endoformulae}) holds for all  $p\in X$ with $1\le l(p)\le n-1$. Consider $p\in X_n$. We set
\[\mu_p=f(p)-\sum_{p=p_1p_2\cdots p_r, r\ge2\atop p_i\in P, i=1,\cdots, r}F(\mu_{p_1}\Box \mu_{p_2}\Box\cdots\Box \mu_{p_r}).
\]
Again (\ref{endoformulae}) holds automatically for $p$ and $\mu_p$ is uniquely determined. The only left is to show that  $\mu_p\in \p_{f(e_{s(p)}),f(e_{t(p)})}$. It follows from the proof of Propostion \ref{prop-coalgmap-to-datum} easily that
$\Delta(\mu_p)= f(e_{s(p)})\otimes \mu_p + \mu_p\otimes f(e_{t(p)})$, this just says $\mu_p$ is a primitive element and the proof is completed.
\end{proof}

\begin{exm} Let $Q$ be a subquiver of $Q'$ and $\iota\colon kQ^c \longrightarrow kQ'^c$ the obvious embedding. Then the trans-datum $\mathbbm{t}_\iota\in \Omega(Q, Q')$ is given by $(\mathbbm{t}_\iota)_{0}(e_i) = e_i$ for any $i\in Q_0$, $(\mathbbm{t}_\iota)_\alpha = \alpha$ for any $\alpha\in Q_1$ and $(\mathbbm{t}_\iota)_p = 0$ for any $p\in Q$ with $l(p)\ge 2$. In particularly, the trans-datum $\mathbbm{t}_{\id_{kQ^c}}\in \Omega(Q)$ corresponding to the identity map of $kQ^c$ is denote by $\mathbbm{1}_Q$, or simply $\mathbbm{1}$ when there is no confusion.
\end{exm}

\begin{rem} The correspondence $\mathbbm{t}$ is given by using induction. It comes to a naive but interesting question: whether there exists an explicit reciprocity formula for (\ref{endoformulae}). Such a formula is helpful in finding the inverse element of an automorphism of a path coalgebra, which could be useful in the study of Jacobian conjecture.
\end{rem}

\begin{rem} One defines \textbf{$n$-truncated trans-datum} for quivers $Q$ and $Q'$ and any integer $n\ge 1$, which characterizes the coalgebra morphisms $\mathrm{Coalg}(kQ^c_{\le n}, kQ'^c_{\le n})$. By definition an $n$-truncated trans-datum from $Q$ to $Q'$ is a pair $\mu = (\mu_0,\mu_\splus)$, where $\mu_0 = \mathrm{Coalg}(kQ_0, kQ'_0)$ is a coalgebra morphism and $\mu_\splus = \{\mu_p\in \p_{\mu_0(e_{s(p)}),\mu_0(e_{t(p)})}\}_{p\in P_{\le n}}$, a family of primitive elements in $kQ'^c$ indexed by $P_{\le n}$, here $P_{\le n}$ denotes the nontrivial paths in $Q$ with length less than or equal to $n$.

Denote by $\Omega(Q, Q'; n)$ the set of  $n$-truncated trans-datum from $Q$ to $Q'$. Similar to Theorem \ref{thm-monomial-extension}, we have a bijection from $\Omega(Q, Q'; n)$ to $\mathrm{Coalg}(kQ^c_{\le n}, kQ'^c_{\le n}))$. Obviously each $\mu \in \Omega(Q, Q'; n)$ gives rise to an element $\mu$ in $\Omega(Q, Q')$ with $\mu_p = 0$ for all path $p$ in $Q$ with length strictly greater than $n$. In other words, any coalgebra homomorphism from $kQ^c_{\le n}$ to $kQ'^c_{\le n}$ extends to a coalgebra homomorphism from $kQ^c$ to $kQ'^c$. This interesting result has a more generalized version.
\end{rem}

\begin{prop}\label{prop-coalg-mor-extension} Let $Q, Q'$ be quivers and $D\subseteq
kQ^c$ a subcoalgebra. Then any $f\in\mathrm{Coalg}(D,kQ'^c)$
extends to a coalgebra homomorphism from $kQ^c$ to $kQ'^c$.
\end{prop}

\begin{proof} First of all, we can extend $f|_{D_0}$ to a coalgebra homomorphism $\mu_0\colon kQ_0\longrightarrow kQ'_0$ easily by setting  $\mu_0(e_i)=f(e_i)$ for $e_i\in D$ and $\mu_0(e_i)=0$ for $e_i\notin D$. Thus $f$ extends to a coalgebra homomorphism $f_0\colon D + kQ_0\longrightarrow kQ'^c$ with $f|_{kQ_0} = \mu_0$.

Let $S=\{(E, g)| D+ kQ_0\subseteq E\subseteq kQ'^c, g\colon E\longrightarrow kQ^c, f|_D = g|_D\}$ be the set of pairs $(E, g)$, where $E$ is subcoalgebra of $kQ^c$ containing $D+kQ_0$ and $g\colon E\longrightarrow kQ'^c$  a colagebra homomorphism extending $f$. The above argument says that $S$ is nonempty. We may define a partial ordering $\preceq$ on $S$, by definition $(E,g)\preceq (E',g')$ if and only if $E\subseteq E'$ and $g = g'|_E$.

 Let $(E_1,g_1)\preceq (E_2,g_2)\preceq\cdots\preceq(E_n,g_n)\preceq\cdots $ be an arbitrary ascending chain in $S$ with respect to the ordering $\preceq$. We set $E=\bigcup_{n\ge1}E_n$. It is easy to show that $E$ is a subcoalgebra of $kQ^c$ which contains all $E_n$'s. We can also define a coalgebra homomorphism $g\colon E\longrightarrow kQ'^c$ as follows. For any $c\in E$, there exists some $n$ with $c\in E_n$, then we set $g(c) = g_n(c)$. Note that $g(c)$ does not depend on the choice of $n$. Thus we have obtained a $(E, g)\in S$ with $(E_n, g_n)\preceq (E, g)$ for all $n\ge 1$.

By Zorn's Lemma, there exists some maximal element $(E, g)$ in $S$.
We claim that $E = kQ^c$. For this it suffices to show that $kQ_{\le n}\subseteq E$ for all $n\ge 0$. We use induction on $n$.

By definition $kQ_0\subseteq E$. Now suppose that $kQ_{\le n-1}\subseteq E$.
By Theorem \ref{thm-monomial-extension}, we may assign each $p \in Q_{\le n-1}$ a primitive element $\mu_{p}\in \p_{f(e_{s(p)}),f(e_{t(p)})}$ in a unique way, such that $(\ref{endoformulae})$ holds for any $p\in Q_{\le n-1} $.

Assume that we have some $p\in Q_n$ such that $p\notin E$. We set $\mu_p = 0$ and $\tilde{E} = E \oplus \mathbb{K} p$. We can extend $g$ to a linear map $\tilde g\colon \tilde{E} \longrightarrow kQ'^c$ by setting
\[\tilde{g}(p) = \sum_{p= p_1p_2\cdots p_r\atop p_i\in Q} F(\mu_{p_1}\Box\mu_{p_2}\Box\cdots\Box\mu_{p_r}).\]
$\tilde{g}$ is a coalgebra morphism since it is when restricted to the subcoalgebra $E$ and to the subcoalgebra generated by $p$. Hence $(E,g)\precneqq (\tilde{E},\tilde{g})$, contrary to the maximality of $(E,g)$. Therefore $kQ_n\subseteq D$, and the proof is completed by induction.
\end{proof}

Combined with Lemma \ref{corad}, we draw the following consequence, a special case of which plays a very important role in the classification of not necessarily coradically graded pointed Hopf algebras; see for instance \cite{hyz}. In fact, this is also one of our main motivations to start with the present work.

\begin{thm}\label{thm-large-auto-ext} Let $Q$ be a finite quiver and $D\subseteq kQ^c$ a large subcoalgebra.  Then any automorphism of $D$
 extends to an automorphism of $kQ^c$. Particularly, $\aut(D)$ is a subquotient group of $\aut(kQ^c)$, here $\aut(kQ^c)$ and $\aut(D)$ are the automorphism group of $kQ^c$ and $D$ respectively.
\end{thm}

\begin{rem} Note that the theorem is not true in general if $D$ is not a large subcoalgebra. For example, let $Q$ be the quiver \[\xymatrix{ 1\bullet \ar@<+0.5ex>[r]^{\alpha} \ar@<-0.5ex>[r]_{\beta} &\bullet2 &\bullet3 \ar[l]_{\gamma} },\] and $kQ^c$ the path coalgebra. Let $D\subseteq kQ^c$ be the subcoalgebra spanned by $\{e_1, e_2, e_3, \alpha, \gamma\}$. Let $\sigma\in\aut(D)$ given by $\sigma(e_1)=e_3$, $\sigma(e_2)=e_2$, $\sigma(e_3)=e_1$, $\sigma(\alpha)=\gamma$ and $\sigma(\gamma)=\alpha$. Clearly $\sigma$ does not extend to an automorphism of $kQ^c$.
\end{rem}

\subsection{Composition of trans-data}
Since trans-data correspond to path coalgebra homomorphisms bijectively, the composition of coalgebra homomorphisms will induced a composition of trans-data. In this subsection we will give the composition map explicitly.

Let $Q, Q'$ be quivers and $Q^\tau, Q'^\tau$ the augmented quivers. As before $P$ and $P^\tau$ denote the set of nontrivial paths in $Q$ and $Q^\tau$ respectively. Let $\mu = (\mu_0, \mu_\splus)\in \Omega(Q,Q')$ be a trans-datum. $\mu$ gives rise to an element $\bar\mu\in \Omega(Q, Q'^\tau)$ in an obvious way, say $ \bar\mu_0 = \mu_0$, $\bar\mu^0 =0$ and
$\bar\mu^1_p = \mu_p$ for all $p\in P$, here each $\mu_p$ is viewed as an element in $kQ'^\tau_1$. The following result is easily deduced from the definition and we omit the proof.

\begin{lem} Let $Q, Q'$ be quivers and $\mu\in \Omega(Q,Q')$. Then we have $\mathbbm{f}_\mu = F\circ f_{\bar\mu}$.
 \[\xymatrix{
 kQ^c \ar @{->}[rr]_{\mathbbm{f}_{\bar\mu}} \ar @{->}@/^1.5pc/[rrrr]^{\mathbbm{f}_\mu} &
& k(Q'^\tau)^c  \ar @{->}[rr]_{F}  &
& kQ'^c }
\]
\end{lem}

 We may also lift $\mu$ to $\tilde\mu\in \Omega(Q^\tau, Q')$ in the following way: $\tilde\mu_0=\mu_0$, $\tilde\mu_p=\mu_p$ for $p\in P$, $\tilde\mu_{e_{s,t}}=\mu_0(e_s-e_t)$ for any $s, t\in Q_0$, $\tilde\mu_{e_{s,s(p)}p} = -\cmu_p\mu_0(e_s - e_{t(p)})$ for $p\in P$ and $\tilde\mu_q = 0$ for any other $q\in P^\tau$. Recall that $\cmu_p$ denotes the coefficient in the expression  $\mu_{p}^0 = \cmu_{p} \mu_0(e_{s(p)} - e_{t(p)})$.

\begin{lem} Keep the above notation. Then $\mathbbm{f}_{\tilde\mu} = \mathbbm{f}_{\mu}\circ F$.
 \[\xymatrix{
k(Q^\tau)^c  \ar @{->}[rr]_{F}  \ar @{->}@/^1.5pc/[rrrr]^{\mathbbm{f}_{\tilde\mu}} &
& kQ^c \ar @{->}[rr]_{\mathbbm{f}_\mu}  & & kQ'^c}
\]
\end{lem}

\begin{proof}
We need only to prove $\mathbbm{f}_{\tilde\mu}(q) = \mathbbm{f}_{\mu}(F(q))$ for any $q\in P^\tau$ and this can be checked case by case directly.

\noindent{\emph{Case 1.}} If $q \in P$ or $q$ has the form $e_{i_1,i_2}e_{i_2,i_3}\cdots e_{i_{n-1},i_n}$, then it is easy to check that  $\mathbbm{f}_{\tilde\mu}(q) = \mathbbm{f}_{\mu}(F(q))$.

\noindent \emph{Case 2.} Let $q = pe_{i_1,i_2}e_{i_2,i_3}\cdots e_{i_{n-1},i_n}$ with $p\in P$. Then by definition \begin{align*}\mathbbm{f}_{\tilde\mu}(q)
 =&\sum_{q=q_1q_2\cdots q_m\atop q_i\in P^\tau} F(\tilde\mu_{q_1}\Box\cdots\Box \tilde\mu_{q_m})\\
 =&\sum_{p=p_1p_2\cdots p_m\atop p_i\in P }F(\tilde\mu_{p_1}\Box\cdots\Box \tilde\mu_{p_m}\Box \tilde\mu_{e_{i_1,i_2}}  \Box \cdots \Box \tilde\mu_{e_{i_{n-1},i_n}})\\
 =& \sum_{p=p_1p_2\cdots p_m\atop p_i\in P }F(\tilde\mu_{p_1}\Box\cdots\Box \tilde\mu_{p_m})\\
 =& \sum_{p=p_1p_2\cdots p_m\atop p_i\in P }F(\mu_{p_1}\Box\cdots\Box \mu_{p_m})\\
 =& \ \mathbbm{f}_\mu(p) = \mathbbm{f}_\mu(F(q)).
\end{align*}

\noindent \emph{Case 3.} Let $q = e_{s,s(p)}p$ with $p\in P$.
\begin{align*}\mathbbm{f}_{\tilde\mu}(q)
 =&\sum_{q=q_1q_2\cdots q_m\atop q_i\in P^\tau} F(\tilde\mu_{q_1}\Box\tilde\mu_{q_2}\Box\cdots\Box \tilde\mu_{q_m})\\
 =&\sum_{p=p_1p_2\cdots p_m\atop p_i\in P }(F(\tilde\mu_{e_{s,s(p)}p_1}\Box\tilde\mu_{p_2}\Box\cdots\Box \tilde\mu_{p_m})
 +F(\tilde\mu_{e_{s,s(p)}}\Box\tilde\mu_{p_1}\Box\tilde\mu_{p_2}\Box\cdots\Box \tilde\mu_{p_m}))\\
 = & - \sum_{p=p_1p_2\cdots p_m\atop p_i\in P } \cmu_{p_1}F(\tilde\mu_{e_{s,t(p_1)}}\Box\tilde\mu_{p_2}\Box\tilde\mu_{p_3}\Box\cdots\Box \tilde\mu_{p_m}) \\
 & + \sum_{p=p_1p_2\cdots p_m\atop p_i\in P }F(\tilde\mu_{e_{s,s(p)}}\Box\tilde\mu^0_{p_1}\Box\tilde\mu_{p_2}\Box\cdots\Box \tilde\mu_{p_m})\\
 & + \sum_{p=p_1p_2\cdots p_m\atop p_i\in P }F(\tilde\mu_{e_{s,s(p)}}\Box\tilde\mu^1_{p_1}\Box\tilde\mu_{p_2}\Box\cdots\Box \tilde\mu_{p_m}).
 \end{align*}
Now by definition of $F$ we have
 \[
    F(\tilde\mu_{e_{s,s(p)}} \Box \tilde\mu^1_{p_1} \Box \tilde\mu_{p_2} \Box
    \cdots \Box \tilde\mu_{p_m}) = -F(\mu^1_{p_1} \Box \mu_{p_2} \Box
    \cdots \Box \mu_{p_m})
 \] and
 \[
   F(\tilde\mu_{e_{s,s(p)}} \Box \tilde\mu^0_{p_1} \Box \tilde\mu_{p_2} \Box
   \cdots \Box \tilde\mu_{p_m})
   = \cmu_{p_1}\cmu_{p_2} \cdots \cmu_{p_m} \tilde\mu_{e_{s,s(p)}}.
 \]
Hence we have
\begin{align*}
   & -F(\tilde\mu_{e_{s,t(p_1)}} \Box \tilde\mu_{p_2} \Box \tilde\mu_{p_3}
     \Box \cdots \Box \tilde\mu_{p_m}) + F(\tilde\mu_{e_{s(p_1), t(p_1)}} \Box
     \tilde\mu_{p_2} \Box \cdots \Box \tilde\mu_{p_m})\\
  =& -F(\tilde\mu_{e_{s,t(p_1)}} \Box \mu^0_{p_2} \Box \mu^0_{p_3}
     \Box \cdots \Box \mu^0_{p_m}) + F(\tilde\mu_{e_{s(p_1), t(p_1)}} \Box
     \mu^0_{p_2} \Box \cdots \Box \mu^0_{p_m}) \\
  =& - \cmu_{p_2}\cmu_{p_3}\cdots\cmu_{p_m} \tilde\mu_{e_{s,t(p_1)}} +
   \cmu_{p_2}\cmu_{p_3}\cdots\cmu_{p_m} \tilde\mu_{e_{s(p_1), t(p_1)}} \\
  =& - \cmu_{p_2}\cmu_{p_3}\cdots\cmu_{p_m} \tilde\mu_{e_{s, s(p_1)}}.
\end{align*}
Therefore we obtain that
\begin{align*}
      \mathbbm{f}_{\tilde\mu}(q)
    = & \sum_{p=p_1p_2\cdots p_m\atop p_i\in P} (-F(\mu^0_{p_1} \Box \mu_{p_2} \Box
        \cdots \Box \mu_{p_m}) -F(\mu^1_{p_1} \Box \mu_{p_2} \Box
        \cdots \Box \mu_{p_m})) \\
    = & \sum_{p=p_1p_2\cdots p_m\atop p_i\in P} -F(\mu_{p_1} \Box \mu_{p_2} \Box
        \cdots \Box \mu_{p_m}) \\
    = & -\mathbbm{f}_\mu(p) = \mathbbm{f}_\mu(F(q))
\end{align*}
 The same argument also works in the case $q=e_{s,s(p)}pe_{t(p), i_1}e_{i_1,i_2}\cdots e_{i_{n-1},i_n}$.

\noindent \emph{Case 4.} Let $q= q'e_{t(q'), s(p)}p$ with $q'\in P^{\tau}$ and $p\in P$.
\begin{align*}\mathbbm{f}_{\tilde\mu}(q)
 =&\sum_{q=q_1q_2\cdots q_m\atop q_i\in P^\tau} F(\tilde\mu_{q_1}\Box\tilde\mu_{q_2}\Box\cdots\Box \tilde\mu_{q_m})\\
 =&\sum_{q'=q'_1q'_2\cdots q'_l \atop {p=p_1p_2\cdots p_n \atop q'_j\in P^\tau, p_i\in P }}
    \left\{F(\tilde\mu_{q'_1}\Box\cdots \Box \tilde\mu_{q'_l} \Box \tilde\mu_{e_{t(q'), s(p)}}\Box \tilde\mu_{p_1}\Box \cdots\Box \tilde\mu_{p_m}) \right. \\
  & \qquad\quad  \left. - F(\tilde\mu_{q'_1}\Box\cdots \Box \tilde\mu_{q'_l} \Box \tilde\mu_{e_{t(q'), s(p)}p_1}\Box \tilde\mu_{p_2}\Box \cdots\Box \tilde\mu_{p_m})\right\}.
\end{align*}
Note that
\begin{align*}
   & F(\tilde\mu_{q'_1}\Box\cdots \Box \tilde\mu_{q'_l} \Box \tilde\mu_{e_{t(q'), s(p)}}
   \Box \tilde\mu_{p_1}\Box \cdots \Box \tilde\mu_{p_c})\\
 = & F(\tilde\mu_{q'_1}\Box\cdots \Box \tilde\mu_{q'_l} \Box \tilde\mu_{e_{t(q'), s(p)}}\Box \tilde\mu^0_{p_1}\Box \cdots \Box \tilde\mu^0_{p_c}) \\
 = & F(\tilde\mu_{q'_1}\Box\cdots \Box \tilde\mu_{q'_l} \Box \mu^0_{p_1}\tilde\mu_{e_{t(q'), t(p_1)}}\Box \tilde\mu^0_{p_2}\Box \cdots \Box \tilde\mu^0_{p_c}),
\end{align*}
it follows that $\mathbbm{f}_{\tilde\mu}(q) = 0 = \mathbbm{f}_\mu(F(q))$ in this case. The same argument works for the case  $q= q'e_{t(q'), s(p)}pe_{t(p),s(q'')}q''$ with $q', q''\in P^{\tau}$ and $p\in P$.

Now we have exhausted all cases and the proof is completed.
\end{proof}

\begin{defn} Let $Q, Q', Q''$ be quivers, $\mu\in\Omega(Q,Q')$ and $\nu\in\Omega(Q',Q'')$. The \textbf{composition} of $\mu$ and $\nu$, denoted by $\nu\circ \mu$, is defined to be the trans-datum in $\Omega(Q, Q'')$ given by $(\nu\circ \mu)_0=\nu_0\circ\mu_0$, and for each nontrivial path $p$, \begin{equation}\label{compositionformula} (\nu\circ \mu)_p = \sum\limits_{p=p_{1}p_{2}\cdots p_{n} \atop {p_i\in P }} \tilde\nu(\mu_{p_{1}}\Box\cdots\Box\mu_{p_{n}}),\end{equation}
here each $\mu_{p_{1}}$ is viewed as an element in $kQ'^\tau_1$ and $\tilde\nu\in \Omega(Q'^\tau, Q'')$ is the lift of $\nu$ as in the above lemma.
\end{defn}

Before stating our main theorem in this section, we give a result in special case.

\begin{lem} Let $Q, Q', Q''$ be quivers, $\mu\in\Omega(Q,Q')$ and $\nu\in\Omega(Q',Q'')$.
Assume that $\mu^0 = 0$. Then $ \mathbbm{f}_\nu\circ \mathbbm{f}_\mu = \mathbbm{f}_{\nu\circ\mu}$.
\end{lem}
\begin{proof}
Since $\mu_{p_1}\Box \mu_{p_2}\Box\cdots\Box \mu_{p_m}$ is a linear combination of paths in $kQ'$, by definition we have $F(\mu_{p_1}\Box\cdots\Box \mu_{p_m}) = \mu_{p_1}\Box\cdots\Box \mu_{p_m}$ and $\tilde\nu(\mu_{p_1}\Box\cdots\Box \mu_{p_m}) = \nu(\mu_{p_1}\Box\cdots\Box \mu_{p_m})$. Now the lemma follows easily by direct calculation. In fact, we have
\begin{align*}
   &(\mathbbm{f}_\nu\circ \mathbbm{f}_\mu )(p)\\
  =& \sum_{p=p_1p_2\cdots p_m, p_i\in P}\mathbbm{f}_\nu(\mu_{p_1}\Box\mu_{p_2}\Box\cdots\Box\mu_{p_m} )\\
  =& \sum_{p=p_1p_2\cdots p_m\atop p_i\in P}\sum_{1\le i_1< \cdots<i_r<m}
          F(\nu(\mu_{p_1}\Box\cdots\Box\mu_{p_{i_1}})\Box  \cdots\Box \nu(\mu_{p_{i_r+1}}\Box\cdots\Box\mu_{p_m})          )
\end{align*}
and
\begin{align*}
 &\mathbbm{f}_{\nu\circ\mu} (p)\\
=&\sum_{p=p_1p_2\cdots p_m\atop p_i\in P} F((\nu\circ\mu)_{p_1}\Box (\nu\circ\mu)_{p_2}\Box\cdots\Box (\nu\circ\mu)_{p_m}) \\
=& \sum_{p=p_1p_2\cdots p_m\atop p_i\in P}\sum_{i=1,2,\cdots, m \atop p_i = p_{i1}p_{i2}\cdots p_{ir_i}}
F(\tilde{\nu}(\mu_{p_{11}}\Box\cdots\Box \mu_{p_{1r_1}})\Box \cdots\Box \tilde{\nu}(\mu_{p_{m1}}\Box\cdots\Box \mu_{p_{mr_m}})) \\
=& \sum_{p=p_1p_2\cdots p_m\atop p_i\in P}\sum_{i=1,2,\cdots, m \atop p_i = p_{i1}p_{i2}\cdots p_{ir_i}}
F({\nu}(\mu_{p_{11}}\Box\cdots\Box \mu_{p_{1r_1}})\Box \cdots\Box {\nu}(\mu_{p_{m1}}\Box\cdots\Box \mu_{p_{mr_m}}))
\end{align*}
for each $p\in P$, and therefore $\mathbbm{f}_{\nu\circ\mu} (p) = (\mathbbm{f}_\nu\circ \mathbbm{f}_\mu )(p)$ by comparing the summations.
\end{proof}

Combining the above lemmas, we have the following characterization.

\begin{thm}\label{thm-composion-map} Let $Q, Q', Q''$ be arbitrary quivers and $\mu\in \Omega(Q, Q')$ and $\nu\in \Omega(Q', Q'')$ trans-data. Then $ \mathbbm{f}_\nu\circ \mathbbm{f}_\mu = \mathbbm{f}_{\nu\circ\mu}$. In particular, $(\Omega(Q), \circ)$ is a monoid with identity $\mathbbm{1}$ and $\mathbbm{f}\colon \Omega(Q) \xrightarrow[]{\ \cong\ } \mathrm{Coalg}(kQ^c, kQ^c)$ is an isomorphism of monoids.
\end{thm}

\begin{proof}  The equality follows easily from the commutative diagram
 \[\xymatrix{
 kQ^c \ar @{->}[rr]_{\mathbbm{f}_{\bar\mu}} \ar @{-->}@/^1pc/[rrrr]^{\mathbbm{f}_\mu} &
& k(Q'^\tau)^c  \ar @{->}[rr]|{F}  \ar @{-->}@/_1pc/[rrrr]_{\mathbbm{f}_{\tilde\nu}} &
& kQ'^c \ar @{->}[rr]^{\mathbbm{f}_\nu}  & & kQ''^c},
\]
note that here we use the fact $\nu\circ \mu = \tilde\nu\circ \bar\mu$. The rest part is easy.
\end{proof}

\begin{rem} Since the composition of coalgebra morphisms obeys the associative law, so is the composition of trans-data. Explicitly, let $Q, Q', Q'', Q'''$ be any given quivers and $\mu\in \Omega(Q, Q')$, $\nu\in \Omega(Q', Q'')$ and $\lambda\in \Omega(Q'', Q''')$ be trans-data. Then $ \lambda\circ(\nu\circ\mu) = (\lambda\circ\nu)\circ\mu$. This can also be checked directly from the definition.
\end{rem}

\section{The automorphism group of a path coalgebra}
Homomorphisms of path coalgebras are characterized in terms of trans-data in last section.
In the rest part of this paper, we will restrict our interest to the automorphism group. We begin with the study of
certain subgroups, which we suppose to be helpful in understanding the structure of the whole group.

\subsection{A tower of normal subgroups of $\aut(kQ^c)$}
Recall that for any coalgebra $C$, we denote by $\aut(C)$ the automorphism group of $C$ and set $\aut_0(C)=\{\sigma\in\aut(C)\mid \sigma|_{C_{(0)}}=\id_{C_{(0)}}\}$, where $C_{(0)}$ is the coradical of $C$.

Let $Q$ be an arbitrary quiver. Set $\Omega^*(Q)$ to be the group of multiplicative invertible elements in $\Omega(Q)$, and hence $\Omega^*(Q)\cong\aut(kQ^c)$ under $\f$ and $\mbt$. By Proposition \ref{prop-coalgmap-to-datum}, $$\Omega^*(Q)=\{(\mu_0, \mu_\splus)\mid \mu_0\in \aut(kQ_0), \mu^1|_{kQ_1}\in GL_k(kQ_1)\}.$$ Notice that each $f\in \aut(kQ^c)$ induces a permutation of $Q_0$. We may consider the subgroup $\aut_0(kQ^c)=\{\sigma\in\aut(kQ^c)\mid \sigma(e_i) = e_i\ \forall i\in Q_0\}$. Denote by $\Omega^*_0(Q)$ the subgroup of $\Omega^*(Q)$ corresponding to $\aut_0(C)$. Obviously
 $\Omega_0^*(Q) = \{(\mu_0, \mu_\splus)\in \Omega^*(Q)\mid \mu_0=\id_{kQ_0}\}$.

By an \textbf{automorphism of a quiver} $Q$, we mean a permutation $\sigma$ of $Q_0$ such that $(Q_1)_{s,t}$ and $(Q_1)_{\sigma_s,\sigma_t}$ have the same cardinality for any $s, t\in Q_0$. Denote by $\aut(Q)$ the group of automorphisms of $Q$.

Recall a classical result which says that the automorphism group of a finite dimensional associative algebra is a finite dimensional linear algebraic group. Thus if $Q$ is finite acyclic, then $kQ$ is finite dimensional and hence $\aut(kQ^c)$ is also a finite dimensional linear algebraic group, comparing with Lemma \ref{lem-coalg-to-alg} below. The following result is easy.

\begin{prop}\label{prop-identity-cpt-aut-coalg} Let $Q$ be any quiver and $C=kQ^c$. Then $\aut_0(C)$ is a normal subgroup of $\aut(C)$ and the quotient group $\aut(C)/\aut_0(C) \cong \aut(Q)$. If in addition, $Q$ is finite acyclic, then $\aut_0(C)$ is the identity component of $\aut(C)$.
\end{prop}

\begin{proof} Consider the restriction map $\mathrm{Res}_0 \colon \aut(C)\longrightarrow \aut(C_0)$. By definition, any automorphism of $Q$ extends to an automorphism of $C$ and hence $\mathrm{Res}_0$ is surjective. Clearly $\Ker(\mathrm{Res}_0) =\aut_0(C)$ and it follows that $\aut(C)/\aut_0(C) \cong \aut(Q) $.

Now assume that $Q$ is finite acyclic. To show the connectedness of $\aut_0(C)$ one uses the fact that $GL(V)$ is connected for any finite dimensional vector space $V$. Let $\sigma\in \aut_0(C)$ and $\mu=\mathbbm{t}_\sigma$ the corresponding trans-datum. Now we define a family of trans-data $\mu_t = (\id, (\mu_t)_\splus)$  parameterized by $t\in[0,1]$, where $\mu_t$ is given by setting $(\mu_t)_\alpha= t\mu^0_{\alpha} + \mu^1_{\alpha}$ for each $\alpha\in Q_1$, and $(\mu_t)_p=t\mu_p$ for any $p\in Q_{\ge2}$. Thus $\mu$ transforms continuously to a datum $(\id, \nu_\splus)$ such that $\nu_{\alpha}\in kQ_1$ for $\alpha\in Q_1$ and $\nu_p = 0$ for  $p\in Q_{\ge2}$. Such data form a subgroup of $\aut(C)$ which is isomorphic to $\prod\limits_{s,t\in Q_0}GL((kQ_1)_{s,t})$. The connectedness of $\aut_0(C)$ follows.

If $Q$ is finite acyclic, then $\aut(C)$ is a finite dimensional linear group. We know that $\aut_0(C)$ is a closed subgroup, for "fixing points" is a closed condition. Thus $\aut_0(C)$ is a connected closed normal subgroup, and hence the identity component of $\aut(C)$, which completes the proof.
\end{proof}

For any coalgebra $D$ and each integer $n\ge 1$, we denote by $\aut_n(D)$ the subgroup of $\aut(D)$ consisting of automorphisms which act on $D_{(n)}$ trivially. Clearly, for a quiver $Q$, $\aut_n(kQ^c)$ corresponds to \[\Omega_n^*(Q)=\{(\id_{Q_0}, \mu_\splus)\in \Omega^*(Q)\mid \mu_\alpha= \alpha, \forall \alpha\in Q_1, \mu_p=0, \forall p\in Q, 2\le l(p)\le n\},\]
here $l(p)$ denotes the length of $p$. For consistence of notations, we also set \[\Omega_{1/2}^*(Q)=\{(\id_{Q_0},\mu_\splus)\in \Omega_0^*(Q)\mid \mu^1_\alpha = \alpha, \forall \alpha\in Q_1\}.\]
Now we have a more generalized result of Propostion \ref{prop-identity-cpt-aut-coalg}.

\begin{prop}\label{prop-normal-subgp-chain} Let $Q$ be a quiver and set $C=kQ^c$. Then we have a tower of normal subgroups of $\Omega^*(Q)$ \[\Omega^*(Q)\vartriangleright\Omega_0^*(Q)\vartriangleright\Omega_{1/2}^*(Q)\vartriangleright\Omega_1^*(Q)\vartriangleright\Omega_2^*(Q)
\vartriangleright\cdots,\]
and $\Omega^*(Q)/\Omega_{n}^*(Q)\cong \aut(C_{(n)})$  for each $n\ge 1$ and $\Omega^*(Q)/\Omega_{1/2}^*(Q)\cong \graut(C_{(1)})$. Moreover, we have $\Omega_0^*(Q)/\Omega_{n}^*(Q)\cong \aut_0(C_{(n)})$ and $\Omega_0^*(Q)/\Omega_{1/2}^*(Q)\cong \graut_0(C_{(1)})$, here the group $\graut(C_{(1)})$ consists of graded automorphisms of $C_{(1)}$ and $\graut_0(C_{(1)})$ the subgroup of graded automorphisms fixing $C_0$.
\end{prop}

\begin{proof} By Lemma \ref{corad}, each automorphism of $C$ remains to be an automorphism when restricted to each $C_{(n)}$. Thus the restriction map gives a natural group homomorphism $\mathrm{Res}_n\colon \Omega^*(Q)\longrightarrow \Omega_n^*(Q)$ for each $n\ge 1$. Again Theorem $\ref{thm-large-auto-ext}$ tells us that the map $\mathrm{Res}_n$ is surjective, and clearly $\Ker(\mathrm{Res}_n) = \Omega_{n}^*(Q)$, thus $\Omega^*(Q)/\Omega_{n}^*(Q)\cong \aut(C_{(n)})$.

 We can associate each trans-datum $\mu\in \Omega^*(Q)$ a linear endomorphism \[\Phi_\mu \colon C_{(1)}\longrightarrow C_{(1)},\  e_i\mapsto \mu_0(e_i),\ \alpha\mapsto \mu^1_\alpha.\]  It is direct to show that $ \Phi_\mu $ is a graded coalgebra automorphism of $C_{(1)}$. In fact, $ \Phi_\mu $ is given by the restriction of the automorphism of $C$ which corresponds to the trans-datum $(\mu_0, \mu^1)$.  Thus we get a map $\Phi\colon \Omega^*(Q)\longrightarrow \graut(C_{(1)})$, which is easily shown to be a group homomorphism. Again $\Phi$ is an epimorphism by Theorem $\ref{thm-large-auto-ext}$ and $\Ker(\Phi)=\Omega_{1/2}^*(Q)$, thus $\Omega^*(Q)/\Omega_{1/2}^*(Q)\cong \graut(C_{(1)})$.

For the rest part we use the same argument as in the proof of Proposition \ref{prop-identity-cpt-aut-coalg}.
\end{proof}

\subsection{Factor groups and solvability of $\aut(kQ^c)$}
In the sequel, the factor groups of the above filtration will be discussed in more detail. The quiver $Q$ is assumed to be finite, so the factor groups can be characterized by the dimension.

The following easy lemmas are useful in studying the solvability of the automorphism group of path coalgebras.
Recall that an algebraic group is solvable if and only if it is solvable as an abstract group.

\begin{lem}\label{lem-auto-coalg-factorgp} Assume $Q$ is finite. Then $\Omega_{1/2}^*(Q)/\Omega_1^*(Q)\cong k^{d_1}$ and $\Omega_{n-1}^*(Q)/\Omega_{n}^*(Q)\cong k^{d_n}$ for any $n\ge 2$, where $d_1 =\sum\limits_{s, t\in Q_0\atop s\ne t} |(Q_1)_{s,t}|$ and $d_n = \sum\limits_{s,t\in Q_0}|(Q_n)_{s,t}|\cdot |(Q^\tau_1)_{s,t}|$, and each $k^{d_i}$ denotes the additive group of a $d_i$-dimensional $k$-vector space.
\end{lem}

\begin{proof}
The proof is given by direct calculation. Note that Formula (\ref{compositionformula}) will be quite simplified in this case. Let
$n\ge 2$ and $\mu, \nu$ be in $\Omega_{n-1}^*(Q)$.   Then $ \mu\circ\nu\in \Omega_{n-1}^*(Q)$ and
\[(\mu\circ\nu)_p = \sum_{p =p_1p_2\cdots p_m, p_i\in P}\tilde\mu(\nu_{p_1}\Box \nu_{p_2}\Box \cdots \Box\nu_{p_m}) = \mu_p + \nu_p\]
for all $p\in P$ with length $n$. Thus the quotient group $\Omega_{n-1}^*(Q)/\Omega_{n}^*(Q)$ is given by the additive group of some $k$-vector space. The dimension $d_n$ is computed by
\[d_n = \sum_{p\in Q, l(p)=n }\dim_k \p_{s(p),t(p)} = \sum_{s,t\in Q_0}|(Q_n)_{s,t}|\cdot |(Q_1^\tau)_{s,t}|.\]
The case $\Omega_{1/2}^*(Q)/\Omega_1^*(Q)$ is proved similarly.
\end{proof}
\begin{rem} In fact, for any quiver $Q$ we always have (abstract) group isomorphisms
\[\Omega_{1/2}^*(Q)/\Omega_1^*(Q)\cong \prod_{s,t\in Q_0}\Hom_k(k(Q_1)_{s,t}, \p_{s,t}),\]
and
\[\Omega_{n-1}^*(Q)/\Omega_{n}^*(Q)\cong \prod_{s,t\in Q_0}\Hom_k(k(Q_n)_{s,t}, \p_{s,t})\]
for all $n\ge 2$, where the right hand sides are both viewed as an additive group.
\end{rem}

\begin{lem}\label{lem-graut-c1} Suppose $Q$ is finite. Then $\graut_0(C_{(1)})\cong \prod_{s,t\in Q_0}GL(k, |(Q_1)_{s,t}|)$.
\end{lem}

Combining Proposition \ref{prop-normal-subgp-chain}, Lemma \ref{lem-auto-coalg-factorgp} and Lemma \ref{lem-graut-c1}, we  can calculate the dimension of the automorphism group of path coalgebras.

\begin{prop}\label{prop-dim-auto-coalg} Assume $Q$ is finite. Then $\aut(kQ^c_{\le n})$ is a linear algebraic group with \[\dim(\aut(kQ^c_{\le n})) = \sum_{s,t\in Q_0} |(Q_1^\tau)_{s,t}|\cdot |(P_{\le n})_{s,t}|\] for any $n\ge 1$, here $(P_{\le n})_{s,t} = \{p\in Q\mid s(p)= s ,t(p)=p, 1\le l(p)\le n\}$. If moreover, $Q$ is finite acyclic, then
 \[\dim(\aut(kQ^c))= \sum_{s,t\in Q_0} (|(Q_1)_{s,t}|+1)\cdot |P_{s,t}|.\]
\end{prop}

  Recall that by a \textbf{Schurian quiver} we mean a quiver $Q$ such that $|(Q_1)_{s,t}|\le 1$ for any $s,t\in Q_0$. We end this section with the following characterization.

\begin{thm}\label{thm-schurian-coalg} Let $Q$ be a finite quiver. Then $\aut_0(kQ^c_{\le n})$ is solvable for some $n\ge 2$ if and only if $Q$ is a Schurian quiver, if and only if $\aut_0(kQ^c_{\le n})$ is solvable for all $n\ge 2$; $\aut(kQ^c_{\le n})$ is solvable for some $n\ge 2$ if and only if $Q$ is a Schurian quiver and $\aut(Q)$ is resolvable, if and only if  $\aut(kQ^c_{\le n})$ is solvable for all $n\ge 2$.
\end{thm}

\begin{proof} Given $n\ge 2$, by Proposition \ref{prop-normal-subgp-chain} and Lemma \ref{lem-auto-coalg-factorgp} we have  $\aut_0(kQ^c_{\le n})$ is solvable if and only if $\graut(C_{(1)})$ is solvable. Note that the general linear group $GL(k,d)$ is solvable if and only if $d=1$. Thus by Lemma \ref{lem-graut-c1}, $\graut(C_{(1)})$ is solvable if and only if for each pair $(s,t)$ in $Q_0$, $|(Q_1)_{s,t}|\le 1$, if and only if $Q$ is Schurian.

Now by Proposition \ref{prop-identity-cpt-aut-coalg}, $\aut(Q)$ is solvable  if and only $\aut_0(kQ^c)$ and $\aut(Q)$ are both solvable, if and only if $Q$ is Schurian with $\aut(Q)$ solvable. This completes the proof.
\end{proof}

Consequently we may obtain a dual version of the result, say a classification of truncated path algebras with solvable automorphism group; see Proposition \ref{prop-sol-autgp-trunalg} below.

\section{A Galois theory for path coalgebras}
In this section, we will develop a Galois-like theory for path coalgebras extensions, which aims to give a connection
between the automorphism groups of subcoalgebras of a path coalgebra and certain subgroups of its automorphism group.

Recall that in Propostion \ref{prop-normal-subgp-chain}, we have already shown for any path coalgebra a tower of normal subgroups of its automorphism group induced by the coradical filtration. Our Galois-like theory will generalize this fact. Note that Theorem \ref{thm-galois-cor-pt-coal} and some useful lemmas in this section hold true only for acyclic quivers.

\subsection{The Galois group of a coalgebra extension}
We begin with a more general setup, although we mainly deal with Galois groups of pointed coalgebra extensions.

\begin{defn} Let $E$ be a coalgebra and $D\subseteq E$ a subspace, the set of automorphisms $\{\sigma\in\aut(E)\mid \sigma(d)=d, \forall d\in D\}$ forms a subgroup of  $\aut(E)$, which we call the \textbf{Galois group} of $E$ over $D$ and denote by $\Gal(E/D)$. Conversely, for any $H\le \aut(E)$, we set $\Inv(H)=\{c\in E\mid \sigma(c) =c, \forall \sigma\in H\}$, the subspace of \textbf{fixed points} of $H$.
\end{defn}

The following lemma is classical. The proof is standard and we omit it here.

\begin{lem}\label{lem-galois-correspond} Let $E$ be a coalgebra. Let $H, H'$ be subgroups of $\aut(E)$ and $D, D'$ subspaces of $E$. Then

$(1)$ $H\le H'\Longrightarrow \Inv(H')\subseteq\Inv(H)$ and $D\subseteq D'\Longrightarrow \Gal(E/D')\subseteq\Gal(E/D)$;

$(2)$ $\Inv(\Gal(E/D))\supseteq D$ and $ \Gal(E/\Inv(H))\ge H$;

$(3)$ $\Gal(E/)\circ\Inv\circ\Gal(E/)=\Gal(E/)$ and $\Inv\circ\Gal(E/)\circ\Inv=\Inv$.
\end{lem}

\begin{rem} Note that in general, $\Inv(H)$ is not a subcoalgebra of $E$. For example, let $Q$ be the quiver
$1\bullet\xrightarrow{\ \alpha\ } 2\bullet\xrightarrow{\ \beta\ } 3\bullet$ and $E=kQ^c$ the pathcoalgebra. Assume that there exists some $a\in k$, $a\notin\{ 0, 1\}$. Consider the automorphism $\sigma\in \aut(E)$ given by
$\sigma(e_i) = e_i, i=1, 2, 3$, $\sigma(\alpha) = a\alpha$, $\sigma(\beta) = a^{-1}\beta$ and $\sigma(\alpha\beta)=\alpha\beta$. Let $H=\langle \sigma\rangle$ the subgroup generated by $\sigma$. Then it is direct to check that $\Inv(H) = ke_1\oplus ke_2\oplus ke_3 \oplus k\alpha\beta$, which is not a subcoalgebra.
\end{rem}

Now we turn to the path coalgebra case. Let $Q$ be a quiver and $C=kQ^c$. We are only interested in large subcoalgebras of $C$ and their Galois groups, since each pointed coalgebra can be realized as a large subcoalgebra. Set $\C(C)= \{D\mid  C_{(1)}\subseteq D\subseteq C\}$, the set of large subcoalgebras of $C$. A subgroup $H\le \aut(C)$ is called an \textbf{admissible subgroup} if $H = \Gal(C/D)$ for some $D\in \C(C)$, and  the set of all admissible subgroups of $\aut(C)$ is denoted by $\G(C)$. Set $\aut_n(C)=\Gal(C/C_{(n)})$ for each $n\ge 0$, comparing with the notation $\Omega^*_n(Q)$ in Section 3. Clearly any $H\in \G(C)$ is a subgroup of $\Gal(C/C_{(1)})$.

For $D\in\C(C)$, we have $D= kQ_0\oplus kQ_1\oplus D_{\ge2}$, where $D_{\ge2}=D\cap kQ_{\ge2}$. Recall that each $\mu\in\Omega(Q)$ gives $k$-linear maps $\mu_0\colon kQ\to kQ_0$ and $\mu_\splus\colon kQ\longrightarrow kQ_0\oplus kQ_1$, see Section 2.1. Then we have the following characterization for Galois groups.

\begin{prop}\label{prop-transdata-galgp} Let $Q$ be an arbitrary quiver, $C=kQ^c$ the path coalgebra and $D\in\C(C)$ a large subcoalgebra. Then
\[\mbt(\Gal(C/D))=\{\mu=(\id_{kQ_0}, \mu_\splus)\mid \mu_\alpha = \alpha\ \forall \alpha \in Q_1, \mu_\splus(x) = 0\ \forall x\in D_{\ge2}\}.\] \end{prop}

\begin{proof} Suppose that $\mu_\alpha = \alpha$ for all $\alpha\in Q_1$ and $\mu_\splus(x) = 0$ for all $x\in D_{\ge2}$.
Then by Proposition \ref{endoformulae2},
\[\f_\mu(x) =\sum_{r\ge 1} F(\mu_\splus(x_{(1)})\Box\cdots \Box \mu_\splus(x_{(r)}))= \f_\1(x),
\]
for in this case $\mu_\splus(x) = {\mathbbm 1}_\splus(x)$ for any $x\in kQ$. Recall that $\1$ is the identity trans-datum which corresponds to $\id_{kQ^c}$.

Conversely, suppose that $\sigma\in\Gal(C/D)$ and set $\mu = \mbt_\sigma$. For
$x\in D$, we denote by $l(x)$ the minimal integer such that $x\in D_{(l(n))}$, where $D_{(l(n))}$ is the $l(n)$-th coradical of $D$. Now we prove that $\mu_{(x)} = 0$ for all $x\in D_{\ge2}$ by induction on $l(x)$.

Clearly $\mu_0=\id_{Q_0}$ and $\mu_\alpha =\alpha$ for all $\alpha\in Q_1$. Assume that $\mu_\splus(x) = 0 =\1_\splus(x)$ for all $x\in D_{\ge2}$ with $l(x)< n$. Take any $x\in D_{\ge 2}$ with $l(x)= n$, Proposition \ref{endoformulae2} says that
\[\sigma(x) = \sum_{r\ge 1} F(\mu_\splus(x_{(1)})\Box\cdots \Box \mu_\splus(x_{(r)}))= \mu_\splus(x) + \sum_{r\ge 2} F(\mu_\splus(x_{(1)})\Box\cdots \Box \mu_\splus(x_{(r)})),\]
and \[\f_\1(x) = \sum_{r\ge 1} F(\1_\splus(x_{(1)})\Box\cdots \Box \1_\splus(x_{(r)}))= \1_\splus(x) + \sum_{r\ge 2} F(\1_\splus(x_{(1)})\Box\cdots \Box \1_\splus(x_{(r)})).\]
Note that for $r\ge 2$, we always have \[\sum F(\mu_\splus(x_{(1)})\Box\cdots \Box \mu_\splus(x_{(r)})) = \sum F(\1_\splus(x_{(1)})\Box\cdots \Box \1_\splus(x_{(r)})),\] the reason is that for each term in the sum, either $l(x_{(i)})< n$ for all $i$, or there exists some $x_i$ with $l(x_i)=0$. Since $\sigma\in\Gal(C/D)$, $\sigma(x) = x = \f_\1(x)$, and hence $\mu_\splus(x) = \1_\splus(x) = 0$.
\end{proof}

To give the following key lemma, we need some notation. We define a partial ordering on the set of all paths in $Q$. For each pair of paths $p, q\in Q$, by $p\le q$ we mean that $p$ is a subpath of $q$, and by $p<q$ we mean that $p$ is a proper subpath of $q$.

\begin{lem}\label{lem-inv-is-coalg} Let $Q$ be an acyclic quiver, $C=kQ^c$ and $D$ a large subcoalgebra of $C$. Then $ \Inv(\Gal(C/D))= D$. Consequently, we have a well-defined map $\Inv\colon \G(C)\longrightarrow \C(C)$.
\end{lem}

\begin{proof} Clearly $D\subseteq \Inv(\Gal(C/D))$. We need only to show that $\Inv(\Gal(C/D)) \subseteq D$.

Assume that $x\in \Inv(\Gal(C/D))\setminus D$. We may write $x=\sum_j \lambda_j p_j$, a linear combination of paths, here $p_j\in Q$ and $0\ne\lambda_j\in k$ for each $j$.  Since $x\notin D$, there exists some $i$ such that $p_i\notin D$. We choose such a $p_i$ with maximal length. Clearly $p_i \nleq p_j$ for any other $p_j$, otherwise we have either $p_j\notin D$ with length strictly greater than $p_i$, or $p_j\in D$ and hence $p_i\in D$, which leads to a contradiction in either case. Moreover, by assumption $Q$ is acyclic and $D$ is a large subcoalgebra, we know that $s(p_i)\ne t(p_i)$.

Let $D'$ denote the subcoalgebra generated by $D$ and all $p_j$'s. We define $\sigma\in\aut(D')$ as follows: $\sigma(p_i) = p_i + e_{s(p_i)} -e_{t(p_i)}$; $\sigma(x) = x$ for any $x < p_i$; $\sigma(x)= x$, $\forall j\ne i$, $\forall x\le p_j$; and $\sigma(d)=d$, $ \forall d\in D$. It is direct to check that $\sigma$ is an automorphism of $D'$, and by Theorem \ref{thm-large-auto-ext}, $\sigma$ extends to $\sigma'\in \aut_0(C)$. By construction $\sigma'\in \Gal(C/D)$ and $\sigma'(x)=x +  e_{s(p_i)} -e_{t(p_i)} \ne x$ and hence $x\notin \Inv(\Gal(C/D))$.
\end{proof}

\begin{rem}\label{rem-gal-not-invertible} The lemma does not hold true in general when the quiver $Q$ contains oriented cycles. For example, consider the quiver $ \xymatrix @C=2.5pc {1\bullet \ar@<+0.5ex>[r]^(.5) { \alpha} & \bullet 2\ar@<+0.5ex>[l]^(.5){\beta}}$. Let $C$ be the path coalgebra. Then by direct calculation one shows that $C_{(2)}\subseteq \Inv(\Gal(C/C_{(1)}))$.
\end{rem}

\subsection{Galois extensions}
Let $D, E$ be coalgebras over a field $k$. If $D$ is a subcoalgebra of $E$ with $E_{(1)} \subseteq D$, then we say that $E$ is an \textbf{admissible extension} over $D$, and write as $E/D$.

 Let $Q$ be a quiver, $C=kQ^c$ the path coalgebra and $D\le C$ a large subcoalgebra. Assume $E/D$ is an admissible extension. By the universal property of $C$, there exists an embedding $E\subseteq C$ which is compatible with the embedding $D\subseteq C$. Roughly speaking, $C$ plays a similar role as an algebraic closure of a field in field theory. Based on this fact, by an admissible extension $E/D$, we always mean a intermediate subcoalgebra $D\subseteq E\subseteq C$. The following definition makes sense now.

\begin{defn} Let $D$ be a large subcoalgebra of $C$. An admissible extension $E/D$ is called a \textbf{Galois extension}, if for any $\sigma\in \Gal(C/D)$, $\sigma(E) = E$.
\end{defn}

Note that $C/D$ is a Galois extension for any large subcoalgebra $D$. Applying Lemma \ref{lem-inv-is-coalg}, we have the following Galois-like correspondence.

\begin{prop}\label{prop-galois-cor-path-coalg} Let $Q$ be an acyclic quiver and $C=kQ^c$ the path coalgebra. Then

(1) The maps $\Inv\colon \G(C)\longrightarrow \C(C)$ and $\Gal(C/)\colon \C(C)\longrightarrow \G(C)$ are well defined and inverse to each other.

(2) Assume that $D, E\in \C(C)$ and set $H=\Gal(C/D)$ and $G=\Gal(C/E)$. Then $\Inv(H\cup G) = D\cap E$ and $\Inv(H\cap G) = D+E$, here $H\cup G$ denotes the subgroup of $\aut(C)$ generated by $H$ and $G$.

(3) Let $H$ be in $\G(C)$. Then $H\vartriangleleft \Gal(C/C_{(1)})$ if and only if $\Inv(H)$ is invariant under the action of  $\Gal(C/C_{(1)})$, i.e. $\Inv(H)$ is a Galois extension over $C_{(1)}$; and in this case, $\Gal(C/C_{(1)})/H\cong \Gal(\Inv(H)/C_{(1)}).$
\end{prop}

\begin{proof} (1)  Lemma \ref{lem-inv-is-coalg} shows that $\Inv(\Gal(C/D))=D$ and $\Inv$ is well defined.
We need only to show that $H = \Gal(C/\Inv(H))$ for any $H\in \G(Q)$. This is easy, for each $H\in \G(C)$ is of the form $\Gal(C/D)$ for some $D\in \C(C)$. Thus by Lemma \ref{lem-galois-correspond}, \[H = \Gal(C/D) = \Gal(C/\Inv(\Gal(C/D))) = \Gal(C/\Inv(H)).\]

(2) First we show that $\Inv(H\cup G) = D\cap E$. By Lemma \ref{lem-galois-correspond}, $\Inv(H\cup G) \subseteq \Inv(H)$ and $\Inv(H\cup G) \subseteq \Inv(H)$ and hence $\Inv(H\cup G) \subseteq D\cap E$. Conversely, for any $c\in D\cap E$, $c$ is fixed by any elements in $H$ and $G$ and hence fixed by $H\cup G$, which implies that $c\in \Inv(H\cup G)$. Thus we have proved $\Inv(H\cup G) = D\cap E$.

Next we show $H\cap G = \Gal(C/(D+E))$. Again by Lemma \ref{lem-galois-correspond}, $\Gal(C/(D+E))\subseteq H$ and $\Gal(C/(D+E))\subseteq G$ and hence $\Gal(C/(D+E))\subseteq H\cap G$. Now let $\sigma\in H\cap G$, $\sigma$ fixes $D$ and $E$ and hence fixes $D+E$. It follows that $H\cap G = \Gal(C/(D+E))$ and hence $\Inv(H\cap G) = D+E$.

(3) If $\Inv(H)$ is an invariant subspace under the action of $\Gal(C/C_{(1)})$, then for any $\sigma\in \Gal(C/C_{(1)})$ and any $h\in H$, $(\sigma^{-1}h\sigma)(c) = (\sigma^{-1}(\sigma(c)) = c$ for any $c\in \Inv(H)$, which implies that $\sigma^{-1}h\sigma\in\Gal(C/\Inv(H)) = H$ and hence $H\vartriangleleft\aut_1(C)$.

Conversely, suppose that $H \vartriangleleft \Gal(C/C_{(1)})$. For any $\sigma\in\Gal(C/C_{(1)})$, we claim that $\sigma(\Inv(H))\subseteq\Inv(H)$. Otherwise there exists some $c\in \Inv(H)$ with $\sigma(c)\notin \Inv(H)$. Now there exists some $h\in H$ such that $h(\sigma(c))\ne \sigma(c)$. Thus $ \sigma^{-1}(h(\sigma(c))) \ne \sigma^{-1}(\sigma(c))$, and hence $\sigma^{-1} h\sigma \notin H$, which leads to a contradiction. Similarly $\sigma^{-1}(\Inv(H))\subseteq \Inv(H)$ and hence $\sigma(\Inv(H))=\Inv(H)$.

We are only left to show the isomorphism of Galois groups. Set $D=\Inv(H)$. We have a group homomorphism $f\colon \Gal(C/C_{(1)}) \longrightarrow \Gal(D/C_{(1)})$ induced by restriction, which is well-defined since $D$ is invariant under the action of $\Gal(C/C_{(1)})$. Theorem \ref{thm-large-auto-ext} says that any automorphism of $D$ extends to an automorphism of $C$, which implies that $f$ is surjective. Now $\Gal(C/C_{(1)})/H\cong \Gal(D/C_{(1)})$  follows easily from the fact $\Ker(f) = \Gal(C/D)$.
\end{proof}

\begin{exm} Let $Q$ be the quiver $1\bullet\xrightarrow{\alpha}2\bullet\xrightarrow{\beta}3\bullet\xrightarrow{\gamma}4\bullet$ and $C=kQ^c$ the path coalgebra. There are exactly 5 subcoalgebras of $C$ containing $C_{(1)}$. Say $C^0 = C_{(1)}$, $C^1= C_{(1)}+k\alpha\beta$, $C^2= C_{(1)}+k\beta\gamma$, $C^{1,2}= C_{(2)}$, and $C^{3}=C$.

The automorphism group $\aut_1(C)\cong k^3 = \{(x, y, z)\mid x, y, z\in k\}$, the additive group of 3-dimensional $k$-space, where $(x,y,z)$ corresponds to the automorphism given by $\alpha\beta\mapsto \alpha\beta+x(e_1-e_3)$, $\beta\gamma\mapsto \beta\gamma+ y(e_2-e_4)$ and $\alpha\beta\gamma\mapsto \alpha\beta\gamma -x\gamma + y\alpha + z(e_1-e_4)$. Compare with Proposition \ref{prop-cor-inn-sbgps} and Example \ref{exm-a-type}.

Set $H^0 = \{(x,y,z)\}$, $H^{1} = \{(0,y,z)\}$, $H^{2} = \{(x,0,z)\}$, $H^{1,2} = \{(0,0,z)\}$ and $H^3=\{(0,0,0)\}$. By direct calculation, $\Gal(C/C^*)=H^*$ for any $*\in\{0,1,2,(1,2),3\}$. Pictorially we have the following correspondence.
\[ \xymatrix @C=-2pc{& \C(C)\\
& C \ar @{-}[d]\\
& C_{(1)}+k\{\alpha\beta,\beta\gamma\} \ar@{-}[ld] \ar@{-}[rd]\\
C_{(1)}+k\alpha\beta \ar@{-}[rd] &  &C_{(1)}+k\beta\gamma \ar@{-}[ld] \\
& C_{(1)}  &
}\  \qquad
\xymatrix @C=-1pc{& \G(C)\\
& \{(0,0,0)\} \ar @{-}[d]\\
& \{(0,0,z)\} \ar@{-}[ld] \ar@{-}[rd]\\
\{(0,y,z)\} \ar@{-}[rd] &  &\{(x,0,z)\} \ar@{-}[ld] \\
& \{(x,y,z)\}  &
}\]
\end{exm}

\begin{rem} Recall that Proposition \ref{prop-transdata-galgp} has determined the trans-data given by $\Gal(C/D)$ for
a large subcoalgebra $D$. It is natural to ask for an intrinsic characterization
of admissibility of subgroups of $\aut(C)$, which could be helpful in our Galois-like theory.
For instance, an admissible subgroup $H\in\G(C)$ must be a closed subgroup in the Zariski sense,
for the condition fixing a point is a "closed" condition and therefore a stability group is always
closed. Up to now, no sufficient and necessary condition is known to us.
\end{rem}

\subsection{Fundamental theorem of Galois extensions}
We move to a more general case now. Again we set $C=kQ^c$, the path coalgebra of a given quiver $Q$. Let $D$ be a large subcoalgebra of $D$ and $E/D$ an admissible extension. Set $\C(D, E) = \{M\mid D\subseteq M\subseteq E\}$, the set of intermediate subcoalgebras of $D$ and $E$, and $\G(D, E)=\{\Gal(E/M)\mid M\in \C(D, E)\}$, the set of admissible subgroups of $\aut(E)$.

\begin{lem} Let $Q$ be an arbitrary quiver, $C=kQ^c$ and $D\in\C(C)$. Assume that $E/D$ is a Galois extension and $M\in\C(D,E)$. Then $E/M$ is a Galois extension, and $\Gal(E/M) \cong \Gal(C/M)/\Gal(C/E)$.
\end{lem}
\begin{proof}  Firstly we show that $E/M$ is a Galois extension. By definition, it suffices to show that any for $\sigma\in\Gal(C/E)$, $\sigma(E)\subseteq E$. This follows easily from the facts $\Gal(C/E)\le \Gal(C/D)$ and $E/D$ is a Galois extension.

To show $\Gal(E/M) \cong \Gal(C/M)/\Gal(C/E)$, we use the same argument as in the proof of Proposition \ref{prop-galois-cor-path-coalg}. Consider the group homomorphism $\mathrm{Res}\colon \Gal(C/M)\longrightarrow \Gal(E/M)$
induced by restriction, which is well defined since $E/M$ is a Galosis extension and $E$ is invariant under the action of $\Gal(C/M)$. Applying Theorem \ref{thm-large-auto-ext} again we show that $\mathrm{Res}$ is surjective. Moreover, it is easy to show that $\Ker(\mathrm{Res}) = \Gal(C/E)$ and hence $\Gal(E/M) \cong \Gal(C/M)/\Gal(C/E)$ by the isomorphism theorem of groups.
\end{proof}

Now we have a more general version of Proposition \ref{prop-galois-cor-path-coalg}, which we call the fundamental theorem of Galois extensions.

\begin{thm}\label{thm-galois-cor-pt-coal} Let $Q$ be an acyclic quiver and $C=kQ^c$ the path coalgebra. Let $D$ be in $\C(C)$ and $E/D$ a Galois extension. Then

(1) The map $\Inv\colon \G(D,E)\to \C(D,E)$ is well defined and gives the inverse of the map $\Gal(E/)\colon \C(D,E)\to \G(D,E)$.

(2) Assume that $M, N\in \C(D, E)$ and set $H=\Gal(E/M)$ and $G=\Gal(E/N)$. Then $\Inv(H\cup G) = M\cap N$ and $\Inv(H\cap G) = M+N$, here $H\cup G$ denotes the subgroup of $\aut(E)$ generated by $H$ and $G$.

(3) Let $H$ be in $\G(D,E)$. Then $H\vartriangleleft \Gal(E/D)$ if and only if $\Inv(H)$ is a Galois extension over $D$; and in this case, $\Gal(E/D)/H\cong \Gal(\Inv(H)/D).$
\end{thm}

\begin{proof} (1) and (2) follow directly from Proposition \ref{prop-galois-cor-path-coalg} and the last lemma. We need only to check (3). The argument is similar to the one in the proof of Proposition \ref{prop-galois-cor-path-coalg}(3).

Set $M = \Inv(H)$. First assume that $H\vartriangleleft \Gal(E/D)$. Let $\sigma\in \Gal(C/D)$. We claim that $\sigma(M)\subseteq M$. Otherwise, there exists some $m\in M$ such that $\sigma(m)\notin M$, and hence some $h\in \Gal(C/M)$ such that $h(\sigma(m))\ne \sigma(m)$. Thus $\sigma^{-1}h\sigma(m)\ne \sigma^{-1}\sigma(m) = m$, and hence $\sigma^{-1}h\sigma\notin \Gal(C/M)$.

We denote by $\bar \sigma$ and $\bar h$ the restriction of $\sigma, h$ to $E$ respectively. Thus $\bar \sigma$ and $\bar h$ can be viewed as elements in $\Gal(E/D)$,  since $E/D$ is a Galois extension. Moreover, $\bar h\in H$ by the choice of $h$. Now $\bar\sigma^{-1}\bar h\bar\sigma\notin H$, which is contradict to the assumption $H\vartriangleleft \Gal(E/D)$. Thus $M$ is $\Gal(C/D)$-invariant, hence $M/D$ is a Galois extension.

Conversely, if $M$ is $\Gal(C/D)$-invariant, then clearly $M$ is $\Gal(E/D)$-invariant. Thus $(\sigma^{-1} h \sigma)(m)= \sigma^{-1}(\sigma(m))=m$ for any $\sigma\in \Gal(E/D)$, $h\in H$ and $m\in M$. It follows that $\sigma^{-1} h \sigma\in H$ and hence $H\vartriangleleft \Gal(E/D)$.

To show the group isomorphism we need the map  $\mathrm{Res}\colon \Gal(E/D)\to \Gal(M/D)$, the group homomorphism induced by restriction. $\mathrm{Res}$ is well defined since $M/D$ is a Galois extension. To show $\mathrm{Res}$ is surjective, we use the fact that any $\sigma\in\Gal(M/D)$ extends to some $\sigma'\in \Gal(C/D)$. By assumption $E/D$ is a Galois extension, the restriction of $\sigma'$ to $E$ can be viewed as an element in $\Gal(E/D)$, which maps to $\sigma$ under $\mathrm{Res}$. It is easy to show that $\Ker(\mathrm{Res})=\Gal(E/M)$ and $\Gal(E/D)/\Gal(E/M)\cong\Gal(M/D)$ follows.
\end{proof}

\begin{rem} The requirement that $Q$ is acyclic is essential in giving one to one correspondence between $\G(C)$ and $\C(C)$; see Remark \ref{rem-gal-not-invertible}. For general quivers, one also expects a similar correspondence between certain class of subcoalgebras of $C$ and certain class of subgroups of $\aut(C)$, but what subcoalgebras and what subgroups should be involved is still unclear to us.
\end{rem}

\begin{rem} We mention that even for a Galois extension $E/D$, an automorphism of $D$ may not extend to an automorphism of $E$.
\end{rem}

Let $E/D$ be an admissible extension and $\aut(E;D)=\{\sigma\in\aut(E)\mid \sigma(D)= D\}$. Define $\Res^E_D\colon \aut(E;D)\longrightarrow \aut(D)$ to be the map induced by restriction. Obviously, $\Ker(\Res^E_D) = \Gal(E/D)$. The above remark just says that $\Res^E_D$ is not surjective in general.

Before going further we introduce some notation. Let $E/D$ be an admissible extension and $M\subseteq E$ a large subcoalgebra. Set $\Gal(E/M;D)=\Gal(E/M)\cap \aut(E;D)$. Now we have the following result on automorphism groups of pointed coalgebras, by applying the correspondence given by taking Galois group and fixed points.

\begin{prop} Let $Q$ be an arbitrary quiver and $C=kQ^c$. Let $D\in \C(C)$. Then

(1)  Assume that $E/D$ is an admissible extension such that $E$ is $\aut(C;D)$-invariant. Then $\aut(D)\cong \aut(E;D)/\Gal(E/D)$. If in addition, $E/D$ is a Galois extension, then $\Res^E_D$ is surjective if and only if $E$ is $\aut(C;D)$-invariant.

(2)  $\aut(D)\cong \aut(C;D)/\Gal(C/D)$. Moreover, if $D\subseteq C_{(n)}$ for some $n\ge1$, then $\aut(D)\cong\aut(C_{(n)};D)/\Gal(C_{(n)}/D)$.

(3)  There is a tower of normal subgroups $\aut(D) \rhd \Gal(D/D_{(0)}) \rhd \Gal(D/D_{(1)}) \rhd\cdots$,
 and $\aut(D)/\Gal(D/D_{(n)})\cong\im(\Res^D_{D_{(n)}}) $ for each $n\ge 0$.

(4)  $\frac{\Gal(D/D_{(n)})}{\Gal(D/D_{(m)})} \cong \frac{\Gal(C/C_{(n)};D)}{\Gal(C/(C_{(n)}+ D_{(m)});D)}, \forall m\ge n\ge 0.$
Particularly, for each $n\ge 1$, the quotient group $ \frac{\Gal(D/D_{(n)})}{\Gal(D/D_{(n+1)})}$ is a subquotient group of $\Gal(C_{(n+1)}/C_{(n)})$ and hence abelian.
\end{prop}

\begin{proof} (1) Suppose that $E$ is $\aut(C;D)$-invariant. Then $\Res^E_D$ is surjective, since any $\sigma\in\aut(D)$ extends to some $\sigma'\in\aut(C;D)$, and $E$ is $\sigma'$-invariant means that $\sigma'$ is an automorphism of $E$ when restricted to $E$. Clearly $\Ker(\Res^E_D)= \Gal(E/D)$ and the isomorphism follows.

Conversely, assume that $E/D$ is a Galois extension and $\Res^E_D$ is surjective. We need to show that for any $\sigma\in\aut(C;D)$, $\sigma(E)\subseteq E$. Since $\Res^E_D$ is surjective, then there exists some $h\in\aut(E, D)$ such that $\Res^E_D(h) = \Res^C_D(\sigma)$. Again by Theorem \ref{thm-large-auto-ext}, $h = \Res^C_E(h')$ for some $h'\in\aut(C; E)$. Clearly $h'\in \aut(C;D)$. Now we have $\sigma^{-1}h'\in\Gal(C/D)$. By definition, $E/D$
is a Galois extension implies that $E$ is $\sigma^{-1}h'$-invariant, together with the fact $E$ is $h'$-invariant we obtain that $E$ is $\sigma$-invariant.

(2) This is a special case of (1). Since $C$ and $C_{(n)}$, $n\ge 0$ are all $\aut(C)$-invariant.

(3) Observe that $\aut(D;D_{(n)})=\aut(D)$ for any $n\ge 0$. Applying the isomorphism theorem of groups to $\Res^D_{D_{(n)}}$, we get the required isomorphism.

(4) Note that we have $\Res^C_D$ induces an epimorphism $\Gal(C/C_{(n)};D)\twoheadrightarrow \Gal(D/D_{(n)})$, the reason is as follows. For any $\sigma\in \Gal(D/D_{(n)})$, we extend $\sigma$ to $\sigma'\in\aut(D+C_{(n)})$ by setting $\sigma'|_D = \sigma$ and $\sigma'|_{C_{(n)}}= \id_{C_{(n)}}$, and extend $\sigma'$ to $\sigma''\in\aut(C)$. Clearly $\sigma''\in\Gal(C/C_{(n)};D)$ and $\Res^C_D(\sigma'')=\sigma$. The isomorphism follows easily.

It is direct to show that $\frac{\Gal(C/C_{(n)};D)}{\Gal(C/(C_{(m)});D)}$ is a subgroup of $\frac{\Gal(C/C_{(n)})}{\Gal(C/(C_{(m)}))}$, which is easy shown to be isomorphic to $\Gal(C_{(m)}/C_{(n)})$. Now \[\frac{\Gal(D/D_{(n)})}{\Gal(D/D_{(m)})} \cong \frac{\Gal(C/C_{(n)};D)}{\Gal(C/(C_{(n)}+ D_{(m)});D)}
\cong \frac{\frac{\Gal(C/C_{(n)};D)}{\Gal(C/C_{(m)};D)}}{\frac{\Gal(C/(C_{(n)}+ D_{(m)});D)}{\Gal(C/C_{(m)};D)}},\]
and hence $\frac{\Gal(D/D_{(n)})}{\Gal(D/D_{(m)})}$ is a subquotient group of $\Gal(C_{(m)}/C_{(n)})$. Particularly, $\frac{\Gal(D/D_{(n)})}{\Gal(D/D_{(m)})}$ is abelian, for $\Gal(C_{(n+1)}/C_{(n)})$ is isomorphic to the additive group of some $k$-vector space by the same argument used in the proof of Lemma \ref{lem-auto-coalg-factorgp}. Note that $\Gal(C/C_{(n)})$ is just the same as $\Omega^*_n(Q)$ defined in last section, and therefore $\Gal(C_{(n+1)}/C_{(n)}) \cong \Omega^*_n(Q)/\Omega^*_{n+1}(Q)$.
\end{proof}

Given a large subcoalgebra $D$ of $kQ^c$, we set $D_\splus = kQ_{\ge1}\cap D$ and $(D_\splus)_{s,t} = (kQ_{\ge1})_{s,t}\cap D$ for any $s,t\in Q_0$. Clearly $D= D_0 \bigoplus \bigoplus\limits_{s,t\in Q_0} (D_\splus)_{s,t}$. By applying Proposition \ref{prop-transdata-galgp}, we obtain a generalization of Proposition \ref{prop-dim-auto-coalg}.

\begin{cor} Let $Q$ be a finite quiver, $C=kQ^c$ and $D\subseteq C$ a finite dimensional large subcoalgebra. Assume $\aut(C; D)=\aut(C)$. Then
\[\dim_k(\aut(D)) = \sum_{s,t\in Q_0}|(Q_1^\tau)_{s,t}| |(D_\splus)_{s,t}|.\]
\end{cor}

\section{Automorphisms of a complete path algebra}

In this section we study the automorphism group of the complete path algebras of a finite quiver.
The key point is that for a finite quiver, the automorphism group of the path coalgebra and the one of the complete path algebra are isomorphic, hence the tool we developed for automorphisms of path coalgebras applies in this case.

Notice that most results on path coalgebras in this section hold true for any quiver, while when dealing with algebras, we always require the quiver to be finite.

\subsection{$\aut(\widehat{kQ^a})\cong \aut(kQ^c)$}

Let $\mathrm{Coalg}$ and $ \mathrm{Alg}$ denote the category of coalgebras and the one
of algebras respectively. Sweedler showed that there is a contravariant functor
$()^*\colon \mathrm{Coalg}\longrightarrow \mathrm{Alg}$.
For any coalgebra $C$, the dual algebra $C^*$ has underlying
vector space $\hom_k(C, k)$ and multiplication given by
the convolution map, see \cite[Section 1]{sw}; and for each coalgebra map $f$, $f^*$ is the usual
transpose map of $f$.

Let $A$ be any algebra and set $A^\circ = \{f\in A^*\mid \ker(f)\ \text{contains a cofinite ideal}\}$. Sweedler showed that for any algebra map $f\colon A\longrightarrow B$, the transpose map $f^*$ maps $B^\circ$ into $A^\circ$. By setting $f^\circ\colon B^\circ\longrightarrow A^\circ$ to be the restriction of $f^*$, we get a contravariant functor $()^\circ\colon \mathrm{Alg}\longrightarrow \mathrm{Coalg}$ which is adjoint to $()^*$, see \cite[Theorem 6.0.5]{sw}.

The following crucial result is due to Taft \cite[Proposition 7.1]{ta}.
Recall that a coalgebra $C$ is said to be \textbf{reflexive}
if $(C^*)^\circ= C$.

\begin{lem}\label{lem-coalg-to-alg} Let $C$ be any coalgebra and $A = C^*$ the dual algebra. Then
$()^*$ defined above induces a group monomorphism $()^*\colon \aut(C)\longrightarrow (\aut(A))^{op}$, here $\aut(A)$ is the automorphism group of $A$. If moreover, $C$ is reflexive, then $\aut(A)\cong(\aut(C))^{op}$ under the map $()^*$ and hence $\aut(A) \cong \aut(C)$.
\end{lem}

Note that the last isomorphism follows from the fact that $G\cong G^{op}$ under the map $g\mapsto g^{-1}$ for any group $G$.
Heyneman and Radford showed that if $C$ is a finitary coalgebra, that is $C_{(1)}$ is finite dimensional, then $C$ is reflexive, see \cite[Theorem 4.1.1]{hr}. Thus we have the following consequence.

\begin{cor}\label{cor-anti-iso} Let $Q$ be a finite quiver. Set $C = kQ^c$ and $A= C^*$. Let $D$ be a large subcoalgebra of $C$ and $B = D^* = A/D^\sperp$ the dual algebra. Then
$()^*\colon \aut(D)\longrightarrow \aut(B)$ is an anti-isomorphism of groups. Particularly,
$\aut(\widehat{kQ^a})\cong\aut(kQ^c)\cong\Omega^*(Q)$.
\end{cor}

\begin{rem} An easy consequence is that $\aut(kQ^c)\cong \aut(kQ^a)$
for any finite acyclic quiver $Q$, for in this case $kQ^a=\widehat{kQ^a}$.
It is worth emphasizing that we do not have such an isomorphism in general case when
$Q$ is not acyclic, $()^*$ even does not give a well-defined
map $\aut(kQ^c)\longrightarrow \aut(kQ^a)$. The reason is that for any $\sigma\in\aut(kQ^c)$,
$\sigma^*\in\aut(\widehat{kQ^a})$ induces an automorphism of the subalgebra $kQ^a$ if and only if $\sigma^*(kQ^a)= kQ^a$,
while this does not hold true in general; see Example \ref{polynomial-alg} below.
\end{rem}

Recall that for any finite quiver $Q$, an element $(a_p)_{p\in Q}$ in $\widehat{kQ^a}$ can be written as an infinite sum $\sum_{p\in Q} a_p\bar p $. For any $\sigma\in \aut(C)$, $\sigma^*$ is determined by $\sigma^*(\overline x) = \sum_{p\in Q}(\overline{x}, \sigma(p))\overline p$ for any path $x\in Q$, here $(-,-)$ is the paring given in Section 1.3. Applying the pairing again one obtains $\sigma(p) = \sum_{x\in Q} (\sigma^*(\overline x), p)x$. Consequently, the automorphism of the complete path algebra corresponding to a trans-datum $\mu$ is given by $\f_\mu^*(\overline x) = \sum_{p\in Q}(\overline{x}, \f_\mu(p))\overline p$.

\subsection{Inner automorphisms}

Let $Q$ be an arbitrary quiver and $C=kQ^c$. Consider the following subgroups of $\Omega^*_0(Q)$, say
\begin{gather*}
 \iomega^*(Q) =  \{(\id_{Q_0},\mu_\splus)\mid \exists k_i\in k^\stimes,  i\in Q_0, \mu^1_\alpha = \frac{k_{s(\alpha)}}{k_{t(\alpha)}}\alpha, \forall \alpha\in Q_1,  \mu^1_p = 0, \forall p\in Q_{\ge2} \}, \\
 \iomega_\circ^*(Q) =  \{(\id_{Q_0},\mu_\splus)\mid \mu^1_\alpha = \alpha, \forall \alpha\in Q_1, \mu^1_p = 0, \forall p\in Q_{\ge2} \}, \\
 \iomega_0^*(Q) =  \{(\id_{Q_0},\mu_\splus)\mid \exists k_i\in k^\stimes, \forall i\in Q_0, \mu_\alpha = \frac{k_{s(\alpha)}}{k_{t(\alpha)}}\alpha, \forall \alpha\in Q_1, \mu_p = 0, \forall p\in Q_{\ge2} \},
\end{gather*}which are of special interest to us.
The reason is that calculations within such groups are much simplified, and more importantly, they are quite helpful in understanding the whole automorphism group.

In case $Q$ is finite, $\iomega^*(Q)$ corresponds to the inner automorphism group of $\widehat{kQ^a}$. For this reason, we call elements in  $\iomega^*(Q)$ inner trans-data and the corresponding automorphisms of $kQ^c$ \textbf{inner automorphisms}.
The following result is practical. Recall that in Section 2.1, we have associated $k$-linear maps  $\mu_0$, $\mu^0$, $\mu^1$, $\mu_\splus$ and $\cmu$ to each trans-datum $\mu$ .

\begin{prop}\label{prop-inn-autocoalg-formular} Let $Q$ be an arbitrary quiver and $C=kQ^c$.
\begin{enumerate}
\item[\rm(1)] Given $\mu\in \iomega_\circ^*(Q)$, then for any $x\in C$,
\begin{align*}\f_\mu(x)&= \mu_0(x) + \mu_\splus(x) + \cmu(x_{(2)})x_{(1)} - \cmu(x_{(2)})x_{(1)}\\ &+
\sum_{r \ge 3}\left[\cmu(x_{(3)})\cdots\cmu(x_{(r)}) \right]\cdot\left[ \cmu(x_{(2)})\mu^0(x_{(1)}) +\cmu(x_{(2)}) x_{(1)} -  \cmu(x_{(1)}) x_{(2)} \right].
\end{align*}

\item[\rm(2)] Given $\mu\in\iomega_0^*(Q)$. Suppose there exists some $\{k_i\in k\}_{i\in Q_0}$ such that $\mu_\alpha = \frac{k_{s(\alpha)}}{k_{t(\alpha)}}\alpha$ for all $\alpha\in Q_1$. Then for any $s, t\in Q_0$ and $x\in kQ_{s,t}$, $f(x) = \frac{k_{s}}{k_{t}}x$.
\end{enumerate}
\end{prop}

Again we use the Sweedler's notation $\Delta^{n}(x)=\sum x_{(1)}\otimes \cdots \otimes x_{(n+1)}$ here. For a proof we need only to check the case that $x$ is a path, which can be done by using a similar argument as in Proposition \ref{endoformulae2} and we omit it here.

The subgroup $\f(\iomega^*(Q))$ is called the \textbf{inner automorphism group} of $C$ and denoted by $\inn(C)$.
We also set $\inn_\circ(C) = \f( \iomega_\circ^*(Q))$ and $\inn_0(C)=\f(\iomega^*_0(Q))$.
Now a basic property of $\inn(C)$ follows easily.

\begin{cor}\label{cor-inn-preserving} Let $Q$ be an arbitrary quiver and $C=kQ^c$. Then $\inn(C)\subseteq \aut(C;D)$ for any subcoalgebra $D\subseteq C$.
\end{cor}

\begin{rem}We remark that the converse of the corollary may not hold, say an automorphism $\sigma\in\bigcap\limits_{D\subseteq C}\aut(C;D)$ needs not to be an inner automorphism. For example,
consider the quiver $Q$ given by exactly one vertex with a loop attached. In this case, all subcoalgebras of $kQ^c$ are given by $kQ^c_{(n)}$, $n\ge 0$. Obviously each automorphism of $kQ^c$
gives an automorphism of any $kQ^c_{(n)}$, while $kQ^c$ has no nontrivial inner automorphism.
\end{rem}

\begin{prop}\label{prop-inn-coalg-is-normal} Let $Q$ be an arbitrary quiver. Then $\iomega^*(Q)\lhd\Omega^*(Q)$, $\iomega^*_\circ(Q)\lhd\Omega^*(Q)$ and $\iomega^*(Q) = \iomega^*_\circ(Q) \rtimes \iomega^*_0(Q)$.
\end{prop}
\begin{proof}
We first show that  $\iomega^*_\circ(Q)\lhd \Omega^*(Q)$. For any
$\mu\in\iomega^*_\circ(Q)$ and $\nu\in\Omega^*(Q)$, \[(\mu\circ\nu)_p = \sum_{p=p_1\cdots p_r, p_i\in P} \tilde \mu(\nu_{p_1}\Box\cdots\Box\nu_{p_r}),\]
and the degree $1$ component of $(\mu\circ\nu)_p$ appears only in the term $\tilde\mu(\nu_p)$, and therefore $(\mu\circ\nu)_p - \nu_p \in k(e_{s(p)}-e_{t(p)})$.
Now it is direct to show that for any nontrivial path $p$, \[(\nu^{-1}\circ(\mu\circ \nu))_p - (\nu^{-1}\circ \nu)_p \in kQ_0, \]
which means that $\nu^{-1}\circ \mu\circ \nu\in\iomega_\circ^*(Q)$ and hence $\iomega^*_\circ(Q)\lhd \Omega^*(Q)$. Similarly, one proves that $\iomega^*(Q)\lhd \Omega^*(Q)$.

It is easy to show that $\iomega^*(Q) = \iomega^*_\circ(Q)\iomega^*_0(Q)$, and the last assertion is obvious since $\iomega^*_\circ$ is normal and $\iomega^*_\circ(Q)\cap\iomega^*_0(Q) = \{\1\}$.
\end{proof}

Let $B$ be an algebra. Denote by $B^\stimes$ the group of multiplicative invertible elements in $B$. For any $b\in B^\stimes$, consider the map $\chi_b\colon B\longrightarrow B$,  $\chi_b(x) = bxb^{-1}$ for any $x\in B$. It is direct to show that $\chi_b$ is an automorphism of $B$, and we call it the \textbf{inner automorphism} induced by $b$. We denote by $\inn(B) = \{\chi_b\mid b\in B^\stimes\}$ the inner automorphism group. $\inn(B)$ is known to be a normal subgroup of $\aut(B)$, and the quotient group $\mathrm{Out}(B) = \aut(B)/\inn(B)$ is called the \textbf{outer automorphism group} of $B$. The map $\chi\colon B^\stimes\longrightarrow \inn(B)$ is usually called the characteristic map. Clearly, $\inn(B)\cong B^\stimes/Z(B)^\stimes$, here $Z(B)$ denotes the center of $B$.

Suppose $B$ is the dual algebra of some coalgebra $D$. The map $()^*\colon \aut(D)\longrightarrow \aut(B)$ considered in Lemma \ref{lem-coalg-to-alg} may not be
surjective in general, while any inner automorphism of $B$ is indeed obtained from an automorphism of $D$, just as shown in the following lemma.

\begin{lem}\label{lem-inn-alg-coalg} Let $D$ be a coalgebra and $B = D^*$ the dual algebra. Then there exists a group homomorphism $\Phi\colon \inn(B)\longrightarrow (\aut(D))^{op}$, such that $(\Phi(\sigma))^* = \sigma$ holds for any $\sigma\in\inn(B)$. In particular, if $D$ is reflexive, then $\Phi=()^\circ$.
\end{lem}

\begin{proof} Define $\Phi\colon \inn(B)\longrightarrow \aut(D)$ as follows. For any $g\in B^\stimes$, $\Phi(\chi_g)\colon D\longrightarrow D$ is given by setting $\Phi(\chi_g)(c) = \sum g(d_{(1)})d_{(2)}g^{-1}(d_{(3)})$ for any $d\in D$. It is easy to check that $\Phi(\chi_g)$ is independent of the choice of $g$, and hence $\Phi$ gives a well-defined map, which is easily shown to be a group homomorphism. Moreover, \[(\Phi(\chi_g))^*(f)(d) = f((\Phi(\chi_g))(d)) = \sum g(d_{(1)})f(d_{(2)})g^{-1}(d_{(3)}) = (gfg^{-1})(d) = (\chi_g(f))(d)\]
for any $f\in B$ and $d\in D$, which implies that $ (\Phi(\chi_g))^* = \chi_g$.
\end{proof}

In the rest part of this subsection, $Q$ is assumed to be a finite quiver. Set $C=kQ^c$ and $A=(kQ^c)^* = \widehat{kQ^a}$.
Recall the isomorphism $\aut(C)\cong \aut(A)$. A natural question is to figure out the subgroups $(\inn(C))^*$, $(\inn_\circ(C))^*$ and $(\inn_0(C))^*$ of $\aut(A)$.

To give an answer, we need some notations. Since $kQ^a$ is a subalgebra of $A$, we also use $A_n$ to denote the subspace $kQ_n$, and $A_{\ge n}$ the subspace $\{\sum_{l(p)\ge n} a_p\overline p\}$ for each $n\ge 0$. Thus $\rad(A) = A_{\ge 1}$, and $A_0$ is a subalgebra with $A^\stimes_0 = \{\sum_{i\in Q_0}a_i\overline{e_i}\mid a_i\in k^\stimes, \forall i\}$. We need also the subgroup $A^\stimes_\circ = 1_A + \rad(A)$. Clearly $A^\stimes = A^\stimes_0 + \rad(A)$ and $A^\stimes = A_\circ^\stimes \rtimes A_0^\stimes$. Set $\inn_\circ(A)=\chi(A_\circ^\stimes)$ and $\inn_0(A)=\chi(A_0^\stimes) $.

\begin{prop}\label{prop-cor-inn-sbgps}
Let $Q$ be a finite quiver. Set $C = kQ^c$ and $A=(kQ^c)^*$. Then we have the following one to one correspondences:
\begin{enumerate}
\item[(1)]\quad  \xymatrix @C=2.5pc {\inn_\circ(A) \ar@<+0.4ex>[r]^(.5){ ()^\circ} &\inn_\circ(C) \ar@<+0.4ex>[r]^(.5){ \mathbbm{t}} \ar@<+0.4ex>[l]^(.49){()^*} & \iomega_\circ^*(Q)\ar@<+0.4ex>[l]^(.5){\mathbbm{f}} };

\item[(2)] \quad  \xymatrix @C=2.5pc {\inn_0(A) \ar@<+0.4ex>[r]^(.5){ ()^\circ} &\inn_0(C) \ar@<+0.4ex>[r]^(.5){ \mathbbm{t}} \ar@<+0.4ex>[l]^(.49){()^*} & \iomega^*_0(Q)\ar@<+0.4ex>[l]^(.5){\mathbbm{f}} };

\item[(3)] \quad  \xymatrix @C=2.5pc {\inn(A) \ar@<+0.4ex>[r]^(.5){()^\circ} &\inn(C) \ar@<+0.4ex>[r]^(.5){ \mathbbm{t}} \ar@<+0.4ex>[l]^(.49){()^*} & \iomega^*(Q)\ar@<+0.4ex>[l]^(.5){\mathbbm{f}} }.
\end{enumerate}
\end{prop}

\begin{proof}
(1)\quad First consider $\inn_\circ(A)$ and $\iomega_\circ^*(Q)$. Let $a = 1- \sum_{p\in P} \lambda_p\overline{p}\in 1 + \rad(A)$. Let $\chi_a$ be the inner automorphism of $A$ induced by  $a$ and $\sigma_a = \chi_a^\circ$ the corresponding automorphism of $C$. We claim that $\mathbbm{t}_{\sigma_a} = \mu = (\id_{Q_0}, \mu_\splus)\in \iomega_\circ^*(Q)$ with $\mu_\splus$ given by $\mu^0_p = \lambda_pe_{s(p),t(p)}$ for each $p\in P$. To show this we need the fact \[a^{-1} = 1+(1-a)+(1-a)^2+ \cdots = 1+ \sum_{p\in P}\sum_{p=p_1p_2\cdots p_m\atop p_i\in P}\lambda_{p_1}\lambda_{p_2}\cdots\lambda_{p_m}\overline{p}.\]

 By direct calculation we have for each $x \in P$,
 \begin{align*}\chi_a(\overline{x}) =\ & (1-a) \overline{x} (1+ (1-a) + (1-a)^2 +\cdots) \\
  =\ &\overline{x} -  \sum_{q\in P \atop t(q) = s(x)}\sum_{p\in P, s(p)=t(x)\atop {p=p_1p_2\cdots p_n\atop p_i\in P}} \lambda_q\lambda_{p_1}\lambda_{p_2}\cdots\lambda_{p_m}\overline{qxp} \\
   \ & - \sum_{q\in P \atop t(q) = s(x)} \lambda_q\overline{qx}+ \sum_{p\in P \atop s(p) = t(x)}\sum_{p=p_1p_2\cdots p_m\atop p_i\in P}\lambda_{p_1}\lambda_{p_2}\cdots\lambda_{p_m}\overline{xp}
  \end{align*}
Let $\mathbbm{f}_\mu$ and $\mathbbm{f}_\mu^*$ be the automorphism of $C$ and $A$ corresponding to $\mu$ respectively. We will prove $\chi_a = \mathbbm{f}_\mu^*$. First assume that $x= x_1x_2\cdots x_r\in P$, here each $x_i\in Q_1$. Then
\begin{align*}
 {\mathbbm{f}_\mu^*}(\overline{x})
  = & \sum_{i\in Q_0}(\overline{x}, e_i)\overline{e_i} + \sum_{p\in P}(\overline{x}, \mathbbm{f}_\mu(p))\overline{p} \\
  = & \sum_{i\in Q_0}(\overline{x}, e_i)\overline{e_i} + \sum_{p\in P}\sum_{p=p_1p_2\cdots p_m\atop p_i\in P}
  (\overline{x}, F(\mu_{p_1}\Box\mu_{p_2}\Box\cdots \Box\mu_{p_m})) \overline{p}\\
  = & \sum_{i\in Q_0}(\overline{x}, e_i)\overline{e_i} + \sum_{p_1p_2\cdots p_m\atop p_i\in P}
  (\overline{x_1x_2\cdots x_r}, F(\mu_{p_1}\Box\mu_{p_2}\Box\cdots \Box\mu_{p_m})) \overline{p_1p_2\cdots p_m}.
\end{align*}

By definition of $F$, $(\overline{x_1x_2\cdots x_r}, F(\mu_{p_1}\Box\mu_{p_2}\Box\cdots \Box\mu_{p_m}))= 0$ unless $p_1p_2\cdots p_r = x$ or $p_2p_3\cdots p_{r+1} = x$, thus
\begin{align*}
 \mathbbm{f}_\mu^*(\overline{x})
  =\ &  (\overline{x_1x_2\cdots x_r}, F(\mu_{x_1}\Box\cdots \Box\mu_{x_r})) \overline{x}\\
    & + \sum_{q \in P}
  (\overline{x_1x_2\cdots x_r}, F(\mu_{q}\Box\mu_{x_1}\Box\cdots \Box\mu_{x_r}) \overline{qx}\\
  &  + \sum_{xp_1p_2\cdots p_m\atop p_i\in P}
  (\overline{x_1x_2\cdots x_r}, F(\mu_{x_1}\Box\cdots \Box\mu_{x_r}\Box\mu_{p_1}\Box\cdots \Box\mu_{p_m})) \overline{xp_1p_2\cdots p_m}\\
    & + \sum_{qxp_1p_2\cdots p_m\atop q, p_i\in P}
  (\overline{x_1x_2\cdots x_r}, F(\mu_q\Box\mu_{x_1}\Box\cdots \Box\mu_{x_r}\Box\mu_{p_1}\Box\cdots \Box\mu_{p_m})) \overline{qxp_1p_2\cdots p_m}\\
  =\ & \overline{x} - \sum_{q\in P\atop t(q)=s(x)}\lambda_q\overline{qx} + \sum_{xp_1p_2\cdots p_m\atop p_i\in P} \lambda_{p_1}\lambda_{p_2}\cdots\lambda_{p_m} \overline{xp_1p_2\cdots p_m}\\
   & -  \sum_{qxp_1p_2\cdots p_m\atop q, p_i\in P}\lambda_q\lambda_{p_1}\lambda_{p_2}\cdots\lambda_{p_m}
   \overline{qxp_1p_2\cdots p_m}.
\end{align*}
Note that when we write $qxp_1p_2\cdots p_m$, we require that $t(q) = s(x)$, $t(x)=s(p_1)$ and $t(p_i)=s(p_{i+1})$ for all $1\le i\le m-1$. It follows that $\chi_a(\overline{x})={\mathbbm{f}_\mu^*}(\overline{x}) $ in case $x\in P$.

Similarly, one shows easily that $\chi_a(\overline{e_i})={\mathbbm{f}_\mu^*}(\overline{e_i})  $ for any $i\in Q_0$. Thus $ \chi_a = {\mathbbm{f}_\mu^*}$ and hence $\mathbbm{t}_{\sigma_a}=(\id_{Q_0}, \mu_\splus)$. It follows that  $\iomega_\circ^*(Q)$ corresponds to $\inn_\circ(A)$.

(2)\quad For any $x = \sum_{i\in Q_0}k_i\overline{e_i}\in A_0^\stimes$, denote by $\chi_x$ and $\sigma_x$ the corresponding automorphism of $A$ and $C$ respectively. It is direct to check that $\tau = \mathbbm{t}_{\sigma_x} = (\id_{Q_0}, \tau_\splus)$, where $\tau_\splus$ is given by $\tau_\alpha = \frac{k_{s(\alpha)}}{k_{t(\alpha)}}\alpha$, $\forall \alpha\in Q_1$, and $\tau_p = 0$, $\forall p\in Q_{\ge2}$. In fact, for any $p=\alpha_1\alpha_2\cdots\alpha_n$,  \[\mathbbm{f}_\tau(\overline{p}) = \frac{k_{s(\alpha_1)}}{k_{t(\alpha_1)}} \frac{k_{s(\alpha_2)}}{k_{t(\alpha_2)}} \cdots \frac{k_{s(\alpha_n)}}{k_{t(\alpha_n)}}\overline{\alpha_1\alpha_2\cdots\alpha_n} = \frac{k_{s(p)}}{k_{t(p)}}\overline p = x\overline{p}x^{-1} = \chi_x(\overline p).\]

(3)\quad Now we consider a general case. Let $a=\sum_{i\in Q_0}k_i\bar{e_i}-\sum_{p\in P}\lambda_p\bar{p}\in A$. Note that $a\in A^\stimes$  if and only if $k_i\ne0$ for all $i\in Q_0$. We may write
 \[
 a = (\sum_{i\in Q_0}k_i\overline{e_i})\cdot (1 - \sum_{p\in P} \frac{1}{k_{s(p)}}\lambda_p\overline{p}) = xy,\] here we set $x = (\sum_{i\in Q_0}k_i\overline{e_i})$, $ y = \sum_{p\in P} \frac{1}{k_{s(p)}}\lambda_p\overline{p}$.
 Let $\chi_a$ be the inner automorphism of $A$ induced by  $a$ and $\sigma_a$ the corresponding automorphism of $C$. By (1) and (2) we know that $\tau = \mathbbm{t}_{\sigma_{x}}$ and $ \omega = \mathbbm{t}_{\sigma_{y}}$ are given by
\begin{gather*}
\mathbbm{t}_{\sigma_x} = (\id_{Q_0}, \tau_\splus),\ \tau_\alpha = \frac{k_{s(\alpha)}}{k_{t(\alpha)}}\alpha\ \forall \alpha\in Q_1,\ \tau_p = 0\ \forall p\in Q_{\ge2},
\\
\mathbbm{t}_{\sigma_y} =  (\id_{Q_0}, \omega_\splus),\  \omega_\alpha = \alpha + \frac{\lambda_{\alpha}}{k_{s(\alpha)}}e_{s(\alpha), t(\alpha)}\ \forall \alpha\in Q_1,\  \omega_p = \frac{\lambda_p}{k_{s(p)}}e_{s(p), t(p)}\ \forall p\in Q_{\ge2}.
\end{gather*}

Clearly $\mathbbm{t}_{{\sigma}_a} = (\id_{Q_0}, \mu_\splus)$  for some $\mu_\splus$ and $a = xy$ implies that
 $ \mathbbm{t}_{{\sigma}_a} = \mathbbm{t}_{{\sigma}_{y}} \circ \mathbbm{t}_{{\sigma}_{x}}$.
 By (\ref{compositionformula}), $\mu$ is given by setting $\mu_\alpha = \frac{k_{s(\alpha)}}{k_{t(\alpha)}}{\alpha} + \frac{\lambda_\alpha}{k_{t(\alpha)}} {e_{s(\alpha), t(\alpha)}}$ for each $\alpha\in Q_1$, and $\mu_{p} = \frac{\lambda_p}{k_{t(p)}} {e_{s(p), t(p)}}$ for each $p\in Q_{\ge2}$. The correspondence between $\inn(A)$ and $\iomega^*(Q)$ follows and the proof is completed.
\end{proof}

\begin{rem} Recall that in Proposition \ref{prop-inn-coalg-is-normal}, we have showed the normality of $\inn(C)$ and $\inn_\circ(C)$ and the semiproduct $\inn(C) = \inn_\circ(C) \rtimes \inn_0(C)$. By the above correspondence, the normality of $\inn(C)$ and $\inn_\circ(C)$
follows easily from the normality of $\inn(A)$ and $\inn_\circ(A)$. While for a direct proof of the semiproduct $\inn(A) = \inn_\circ(A) \rtimes \inn_0(A)$, it is not quite obvious and takes some works.
\end{rem}

\subsection{$\aut_0(\widehat{kQ^a})= \inn(\widehat{kQ^a})\aut_\bullet(\widehat{kQ^a})$}
The automorphisms of $C$ which act invariantly on the coideal $C_{\ge 1}$ form an important subgroup $\aut_\bullet(C)$ of $\aut(C)$, say $\aut_\bullet(C) = \{\sigma\in\aut_0(C)\mid \sigma(C_{\ge 1})\subseteq C_{\ge 1} \}$. Consider the following subset of $\Omega^*(Q)$:
\[
 \Omega^{*}_\bullet(Q)=\{(\id_{Q_0},\mu)\in\Omega_0^*(Q)\mid \mu_\splus=\mu^1 \}.
\]
Then we have the following correspondence.

\begin{prop}\label{prop-cor-bullet-sbgp}
Let $Q$ be an arbitrary quiver and $C = kQ^c$. Then
$\mbt(\aut_\bullet(C)) = \Omega^{*}_\bullet(Q)$ and
$\f(\Omega^{*}_\bullet(Q))=\aut_\bullet(C)$.
\end{prop}

\begin{proof}  We need only to show that $\mathbbm{t}(\aut_\bullet(C))\subseteq \Omega^{*}_\bullet(Q)$ and
 $\mathbbm{f}(\Omega^{*}_\bullet(Q))\subseteq \aut_\bullet(C)$, and this follows easily from the construction
 of $\mathbbm{f}$ and $\mathbbm{t}$. The argument is as follows.

Let $\mu = (\id_{Q_0},\mu_\splus)\in \Omega^{*}_\bullet(Q)$. Then by definition, for any $p\in P$ we have
\[\f_\mu(p) = \sum_{p=p_1p_2\cdots p_m\atop p_i\in P} F(\mu_{p_1}
\Box\mu_{p_2}\Box\cdots\Box\mu_{p_m}),\] now $\mu_\splus=\mu^1$ implies
each $F(\mu_{p_1}\Box\mu_{p_2}\Box\cdots\Box\mu_{p_m})\in C_{\ge 1}$ and hence $\f_\mu\in \aut_\bullet(C)$.

Conversely, if $\mu = (\id_{Q_0},\mu_\splus)\notin \Omega^{*}_\bullet(Q)$, then there exists some $p\in P$,
such that $\mu^0_p\ne 0$. We take such a $p$ with minimal length. Then \[\f_\mu(p) = \sum_{p=p_1p_2\cdots p_m\atop m\ge2, p_i\in P} F(\mu_{p_1}
\Box\mu_{p_2}\Box\cdots\Box\mu_{p_m}) + \mu_p,\] since $\mu_{q} = \mu^1_{q}$ for any $q$ with length strictly less than $p$, we know that $\f_\mu(p)\in \mu^0_p + C_{\ge1}$,
and hence $\f_\mu\notin \aut_\bullet(C)$.
\end{proof}

\begin{rem} Note that when we use the notation $C_{\ge 1}$, we have already fixed
 a natural projection $C\twoheadrightarrow C_{(0)}$ of coalgebras,
 or equivalently, we fixed a decomposition $C = C_{(0)} \oplus C_{\ge1}$.
Note that the subgroup $\aut_\bullet(C)$ depends on the choice of the projection $C\twoheadrightarrow C_0$.
\end{rem}

The following result is given by direct calculation.

\begin{prop}\label{prop-decom-aut-pathcoalg} Let $Q$ be an arbitrary quiver. Then $\Omega_0^*(Q) = \Omega^*_\bullet(Q)\iomega^*_\circ(Q)$. Consequently, $\aut_0(kQ^c) = \aut_\bullet(kQ^c)\inn_\circ(kQ^c)$.
\end{prop}

\begin{proof} Take $\mu=(\id, \mu_+)\in \Omega_0^*(Q)$. We will construct inductively some
$\nu\in \iomega_\circ^*(Q)$ such that $\mu\circ\nu\in \Omega_\bullet^*(Q)$. If this has been done,
then $\mu = (\mu\circ\nu)\circ\nu^{-1}\in \Omega^*_\bullet(Q)\iomega^*_\circ(Q)$, and the
conclusion follows.

In fact, set $\nu_0 = \id$ and $\nu_\alpha = \alpha$ for any $\alpha\in Q_1$. For each $p$ with $l(p)=2$,
we set $\nu_p = -\mu^0_p$. Then clearly, \[\sum\limits_{p=p_{1}p_{2}\cdots p_{r} \atop {p_i\in P }} \tilde\mu(\nu_{p_{1}}\Box\cdots\Box\nu_{p_{r}})= \tilde\mu_{\nu_p} + \tilde\mu_p =\mu^1_p\in kQ_1\]
holds whenever $l(p)=2$.

Assume that for each $p$ with
$2\le l(p)< n$, we have fixed a $\nu_p\in k(e_{s(p)} - e_{t(p)})$, such that $\sum\limits_{p=p_{1}p_{2}\cdots p_{r} \atop {p_i\in P }} \tilde\mu(\nu_{p_{1}}\Box\cdots\Box\nu_{p_{r}})\in kQ_1$ holds for any
$p$ with $l(p)<n$. Now for a path $p$ with $l(p)=n$, we set
\[\nu_p=-\mu_p^0 - \sum\limits_{p=p_{1}p_{2}\cdots p_{r} \atop {r\ge 2, p_i\in P }} \tilde\mu(\nu_{p_{1}}\Box\cdots\Box\nu_{p_{r}}).\]
By definition of $\tilde\mu$, it is direct to show that $\nu_p\in k(e_{s(p)} - e_{t(p)})$, and
\[\sum\limits_{p=p_{1}p_{2}\cdots p_{r} \atop {p_i\in P }} \tilde\mu(\nu_{p_{1}}\Box\cdots\Box\nu_{p_{r}})=\mu_p^1\in kQ_1.\]
By induction, we obtain some $\nu\in \iomega_\circ^*(Q)$, such that $(\mu\circ \nu)_p\in kQ_1$ for any nontrivial path $p$, that is $\mu\circ\nu\in\Omega_\bullet^*(Q)$. The proof is completed.
\end{proof}

Now let $D\subseteq C$ be a large subcoalgebra of $C$. By Corollary \ref{cor-inn-preserving}, we have $\inn(C)\le \aut(C;D)$. We set $\inn(D)=\Res^C_D(\inn(C))$, and call it the \textbf{inner automorphism group} of $D$. We also set $\inn_\circ(D)=\Res^C_D(\inn_\circ(C))$ and $\inn_0(D)=\Res^C_D(\inn_0(C))$. We have a generalized version of Proposition \ref{prop-inn-coalg-is-normal}.

\begin{prop}\label{prop-large-coalg-inn-auto-gp} Let $Q$ be a quiver, $C=kQ^c$ and $D\subseteq C$ a large subcoalgebra. Then

(1)\quad $\inn(C)\le \aut(C;D)$, and the restriction map induces isomorphisms of groups $\inn(C)/(\inn(C)\cap\Gal(C/D))\cong \inn(D)$, $\inn_\circ(C)/(\inn_\circ(C)\cap\Gal(C/D))\cong \inn_\circ(D)$ and $\inn_0(C)/(\inn_0(C)\cap\Gal(C/D))\cong \inn_0(D)$;

(2)\quad $\inn(D)^* = \inn(D^*)$, $\inn_0(D)^* = \inn_0(D^*)$ and $\inn_\circ(D)^* = \inn_\circ(D^*)$;

(3)\quad $\inn(D)\lhd\aut(D)$, $\inn_\circ(D)\lhd\aut(D)$ and $\inn(D) = \inn_\circ(D) \rtimes \inn_0(D)$.
\end{prop}

\begin{proof} We need only to prove (3). Recall that we have already shown that $\iomega^*_\circ(Q)\lhd \Omega^*(Q)$ and $\iomega^*(Q)\lhd \Omega^*(Q)$ in Proposition \ref{prop-inn-coalg-is-normal}, and hence $\inn(C)\lhd\aut(C)$ and $\inn_\circ(C)\lhd\aut(C)$. Now $\inn(D)\lhd\aut(D)$ and $\inn_\circ(D)\lhd\aut(D)$ follows from (1).

For the last assertion, it suffices to prove that $ \inn_\circ(D) \cap \inn_0(D)= \{\id_D\}$, which is equivalent to show that for any $\sigma\in \inn_0(C)$ and $\tau\in \inn_\circ(C)$, $\sigma\circ\tau\in \Gal(C/D)$ if and only if $\sigma$ and $\tau$ are both in $\Gal(C/D)$. This is obvious since $D$ is large.
\end{proof}

Consider the decomposition
$D = D_{(0)} \oplus D_\splus$, where $D_\splus = D\cap C_{\ge1}$ of $D$ is a coideal of $D$. Set $\aut_0(D) = \Gal(D/D_{(0)})$ and $ \aut_\bullet(D) =
 \{\sigma\in\aut_0(D)\mid \sigma(D_\splus)\subseteq D_\splus\}$. By a similar argument as in the
 proof of Proposition \ref{prop-coalg-mor-extension}, any
  $\sigma\in \aut_\bullet(D)$ extends to an element in $\aut_\bullet(C)$. Thus
if we set $\aut_\bullet(C;D) = \aut_\bullet(C)\cap\aut(C; D)$, then
 $\aut_\bullet(D) = \Res^C_D(\aut_\bullet(C;D))$. Applying Proposition \ref{prop-large-coalg-inn-auto-gp},
 we have the following characterization of $\aut_0(D)$, generalizing Proposition \ref{prop-decom-aut-pathcoalg}.

\begin{prop}\label{prop-decom-aut-coalg}  Let $Q$ be an arbitrary quiver and $D\subseteq kQ^c$ a large subcoalgebra. Then $\aut_0(D) = \aut_\bullet(D)\inn_\circ(D)$.
\end{prop}

\begin{rem}\label{rem-connectness-coalg} Recall that in Proposition \ref{prop-identity-cpt-aut-coalg},
$\aut_0(D)$ is shown to be the identity component of $\aut(D)$ if $D=kQ^c$ for some finite acyclic quiver $Q$.
A naive question is to ask whether this property holds for any finite dimensional pointed coalgebra $D$.
We mention that the answer is yes, if $D$ is given by a truncated subcoalgebra of a finitary path coalgebra,
for this one uses a similar argument as in the proof of Proposition \ref{prop-identity-cpt-aut-coalg}.

When concerning with an arbitrary finitary pointed coalgebra $D$, the answer is no, the reason is that $\aut_0(D)$ is not connected in general. For instance,
consider the  quiver
\[\xymatrix{ 1\bullet \ar@<+0.5ex>[r]^{\alpha_1} \ar@<-0.5ex>[r]_{\alpha_2} &\bullet2 \ar@<+0.5ex>[r]^{\beta_1} \ar@<-0.5ex>[r]_{\beta_2} &\bullet3 },\]
and the large subcoalgebra $D=k\{e_1, e_2, e_3, \alpha_1,\alpha_2,\beta_1,\beta_2,\alpha_1\beta_1, \alpha_2\beta_2\}$. Direct calculation shows that $\aut_0(D)$ is not connected.

\end{rem}

In the rest part of this subsection, we will focus on algebras and the quiver $Q$ is assumed to be finte. Let $D$ be a large subcoalgebra of $C=kQ^c$, and $B = A/D^\sperp$ the dual algebra of $D$, here $D^\sperp=\{f\in A\mid f(d)= 0, \forall d\in D\}$. Note that $D$ is large if and only if $ D^\sperp\subseteq \rad^2(A)$.

Note that $A_0$ is embedded to $B$ via the map $A_0\hookrightarrow A\twoheadrightarrow B$, and the image is denoted by $B_0$. Then we have a decomposition $B = B_0 \oplus \rad(B)$, which just corresponds to the decomposition $D = D_{(0)} \oplus D_\splus$. Set $\inn_0(B)=\chi(B_0^\stimes)$ and $\inn_\circ(B)=\chi(1+\rad(B))$.  We claim that any inner automorphism of $B$ is induced from an inner automorphism of $A$. The reason is that for any $a\in A$, $a+D^\sperp$ is invertible in $B$ if and only if $a$ is invertible in $A$. This can be summarized as follows.

\begin{lem}\label{lem-inn-quotient-algebra} Let $Q$ be a finite quiver, $C=kQ^c$ the path coalgebra and $A=(kQ^c)^*$. Let $D$ be a large subcoalgebra of $C$ and $B=D^*$ its dual algebra. Then
$\inn(B)=(\inn(D))^*$, $\inn_\circ(B)=(\inn_\circ(D))^*$ and $\inn_0(B)=(\inn_0(D))^*$.
\end{lem}

Similarly, set $\aut_0(B) = \{{\sigma}\in\aut(B)\mid {\sigma}
(\overline{e_i})\in \overline{e_i}+\rad(B), \forall i\in Q_0\}$, and $\aut_\bullet(B)=\{\sigma\in\aut(B)\mid \sigma|_{B_0}=\id_{B_0}\}$.
Denote by $Z^\stimes_{B}(B_0) =\{a\in B^\stimes\mid ab=ba, \forall b\in B_0 \}$ the centralizer of $B_0$ in $B^\stimes$. Then we have the following structural result for $\aut(B)$.

\begin{thm}\label{thm-antiiso-dual-alg} Let $Q$ be a finite quiver, $C = kQ^c$ and $A= C^*$. Let $D\subseteq C$ a large subcoalgebra and $B = D^* = A/D^\sperp$ the dual algebra. Then

(1)  $(\aut_0(D))^* =  \aut_0(B)$.

(2)  $(\aut_\bullet(D))^*= \aut_\bullet(B)$.

(3)  $\inn(B)\lhd\aut(B)$, $\inn_\circ(B)\lhd\aut(B)$ and $\inn(B) = \inn_\circ(B) \rtimes \inn_0(B)$.

(4)  $\aut_0(B) = \inn_\circ(B)\aut_\bullet(B)$ and $\inn(B)\cap \aut_\bullet(B) = \chi(Z^\stimes_{B}(B_0))$.
If in addition, $Z^\stimes_{B}(B_0) = B_0^\stimes$, then $\aut_0(B) = \inn_\circ(B)\rtimes\aut_\bullet(B)$.

(5)  There is an exact sequence of groups \[1\longrightarrow \aut_\bullet(B)/\chi(Z^\stimes_{B}(B_0))\longrightarrow \aut(B)/\inn(B)\longrightarrow\aut(B)/\aut_0(B)\longrightarrow 1.\]
\end{thm}

\begin{proof} (1) Let $\sigma\in\aut(D)$. By definition, for each $i\in Q_0$ we have
 \[\sigma^*(\overline {e_i}) = \sum_{j\in Q_0}(\overline {e_i}, \sigma(e_j))\overline{e_j}
 +\sum_{s\in S, l(s)\ge 1}(\overline{e_i},\sigma(s))\overline s \in \overline{\sigma(e_i)}+\rad(B),\]
it follows that $\sigma^*\in \aut_0(B)$ if and only if $\sigma\in\aut_0(D)$. Hence $(\aut_0(D))^*= \aut_0(B)$.

(2)   We prove that for any $\sigma\in\aut(D)$, $\sigma^*\in \aut_\bullet(B)$ if and only if $\sigma\in \aut_\bullet(D)$.
Assume that $\sigma\in \aut_\bullet(D)$. By definition $\sigma(D_\splus)\subseteq D_{\splus}$ and
hence $(\overline{e_i},\sigma(s))= 0$ for any $s\in S$ with $l(s)\ge 1$ and $i\in Q_0$. Thus
 \[\sigma^*(\overline {e_i}) = \sum_{j\in Q_0}(\overline {e_i}, \sigma(e_j))\overline{e_j}
 +\sum_{s\in S, l(s)\ge 1}(\overline{e_i},\sigma(s))\overline s = \overline {e_i}\] for any $i\in Q_0$, which implies that $\sigma^*\in \aut_\bullet(B)$.

If  $\sigma\notin \aut_\bullet(D)$, then there exists some $x\in D_\splus$ such that $\sigma(x)\notin D_\splus$. Then there exists some $i\in Q_0$ such
that $(\overline{e_i}, x)\ne 0$. Hence $(\sigma^*(\overline{e_i}), x)= (\overline{e_i}, x)\ne 0$, which implies that
$\sigma^*(\overline{e_i})\notin B_0$ and $\sigma^*\notin\aut_\bullet(B)$. Thus we have $(\aut_\bullet(D))^*= \aut_\bullet(B)$.

(3) follows directly from Proposition \ref{prop-large-coalg-inn-auto-gp} and Lemma \ref{lem-inn-quotient-algebra}.

(4) The first assertion follows from (2) and Proposition \ref{prop-decom-aut-coalg}. We next show that $\inn(B)\cap \aut_\bullet(B) = \chi(Z_{B^\stimes}(B_0))$. For any $b\in B^\stimes$, $\chi_b\in \aut_\bullet(B)$ if and only if $\chi_b(\overline{e_i}) = b\overline{e_i}b^{-1} = \overline{e_i}$ for any $i\in Q_0$, if and only if $b\in Z_B(B_0)$, it follows that $\inn(B)\cap \aut_\bullet(B) = \chi(Z^\stimes_{B}(B_0))$.

Since $\inn_0(B)\subseteq \inn(B)\cap \aut_\bullet(B)$, we obtain that $\aut_0(B) = \inn_\circ(B)\aut_\bullet(B)$. If $Z^\stimes_{B}(B_0) = B_0^\stimes$, then $\inn_\circ(B)\cap \aut_\bullet(B) = \{\id\}$, and the semi-product follows.

(5) follows easily from (4) and the second isomorphism theorem of groups, which says that $\inn(B)\aut_\bullet(B)/\inn(B)\cong \aut_\bullet(B)/\inn(B)\cap\aut_\bullet(B)$.
\end{proof}

\begin{rem} The parts (3), (4) and (5) of the theorem should be known to experts. For a direct proof of (4), one needs the following classical result, which can be found in any standard textbook on ring theory, see  for instance, \cite[Section 21]{la}, or  \cite[Theorem 3.4.1]{dk} for finite dimensional algebra case:

Let $R$ be a ring with identity and $1 = x_1 + x_2 +\cdots + x_m = y_1 + y_2+ \cdots + y_n$ be two primitive orthogonal decompositions of identity, then $ m = n $ and there exists $r\in R^\stimes$ and a permutation $\tau$ of $\{1, 2, \cdots, n\}$ such that $x_i = ry_{\tau(i)}r^{-1}$ for all $1\le i\le n$.

Let $\sigma\in\aut_0(B)$. Then we have primitive orthogonal decompositions of identity
$1= \sum_{i\in Q_0}{\overline{e_i}} = \sum_{i\in Q_0} {\sigma}({\overline{e_i}})$. By the above result,
there exists some $b\in B^{\stimes}$ and a permutation $\tau$ of $Q_0$, such
that $\chi_b(\overline{e_i}) = b\cdot\overline{e_i}\cdot b^{-1} = {\sigma}(\overline{e_{\tau(i)}})$ for any $i\in Q_0$. Now
\[\chi_b(\overline{e_i}) =  {\sigma}(\overline{e_{\tau(i)}})\in \overline{e_{\tau(i)}} + \rad(B)\] implies that $\tau = \id_{Q_0}$ and hence $\chi_b(\overline{e_i}) =  {\sigma}(\overline{e_{i}})$.
It follows that $\chi_b^{-1} {\sigma}\in\aut_\bullet(B)$ and
$\aut_0(B) = \inn(B)\aut_\bullet(B)$.

Note that the above argument is only an existence proof, while our proof of Proposition \ref{prop-decom-aut-pathcoalg} is in some extend a constructive one.

\end{rem}

\subsection{The finite dimensional case}
We will apply to finite dimensional case and recover some classical results. Note that when the quiver $Q$ is finite acyclic, the path algebra is finite dimensional and coincides with the complete path algebra, and hence $\aut(kQ^c)\cong\aut(kQ^a)$.

\begin{prop} \label{prop-autogp-path-alg} Let $Q$ be a finite acyclic quiver and $A = kQ^a$ the path algebra.  Then

(1) $\aut_0(A)$ is the identity component of $\aut(A)$ and $\aut(A)/\aut_0(A)\cong\aut(Q)$;

(2) $\inn(A)\cap \aut_\bullet(A)= \inn_0(A)$ and hence $\aut_0(A)/\inn(A) \cong \aut_\bullet(A)/\inn_0(A)$;

(3) $\aut_0(A) = \inn_\circ(A)\rtimes\aut_\bullet(A)$;
\end{prop}

\begin{proof}  Let $C= kQ^c$. It is direct to show that $\aut_0(A) = \{\sigma^*\mid \sigma\in \aut_0(C)\}$. Thus the assertion (1) follows from Proposition \ref{prop-identity-cpt-aut-coalg} together with the fact $\aut(C)$ and $\aut(A)$ are anti-isomorphic.
(2) and (3) follow from Theorem \ref{thm-antiiso-dual-alg}. Note that $Z_{A}(A_0) = A_0$ in case $Q$ is acyclic.
\end{proof}

Moreover, by Theorem \ref{thm-schurian-coalg} we have the following characterization for solvability of the automorphism group of a truncated path algebra.

\begin{prop}\label{prop-sol-autgp-trunalg} Let $Q$ be a finite quiver, $A=kQ^a$ and $J=kQ_{\ge1}$ the graded Jacobson ideal. Then
$\aut_0(A/J^n)$ is solvable for some $n\ge 2$ if and only if $Q$ is a Schurian quiver, if and only if
$\aut_0(A/J^n)$ is solvable for all $n\ge 2$; $\aut(A/J^n)$ is solvable for some $n\ge 2$ if and only if $Q$ is a Schurian quiver with $\aut(Q)$ being solvable, if and only if $\aut(A/J^n)$ is solvable for all $n\ge 2$.
\end{prop}

\begin{rem} The results should be known to experts. For instance, the solvability of the identity component of the automorphism group of a monomial algebra has been discussed in \cite[Corollary 2.23]{gs2}, from which the first part of the above proposition can also be deduced.
\end{rem}

Combined with Proposition \ref{prop-cor-inn-sbgps} and \ref{prop-cor-bullet-sbgp}, we reobtain the following known dimension formula, which is essentially given in \cite[Remark 4.10]{gs2}.

\begin{cor}\label{cor-dim-outer-gp} Let $Q$ be a finite acyclic quiver and $A=kQ^a$. Then  \[\dim(\mathrm{Out}(A)) = \sum_{ s, t\in Q_0} |(Q_1)_{s,t}|\cdot |(Q_{\ge1})_{s,t} | - |Q_0| + r,\] where $r$ denotes the number of connected components of the quiver $Q$.
\end{cor}

\begin{proof} For a proof we use the correspondence given in  Proposition \ref{prop-cor-bullet-sbgp}. One shows that $\dim(\aut_\bullet(A)) =  \sum_{ s, t\in Q_0} |(Q_1)_{s,t}|\cdot |(Q_{\ge1})_{s,t} |$ and $\dim(\inn_0(A)) = |Q_0| - r$. By Proposition \ref{prop-autogp-path-alg}(3) we have
\[\dim(\aut_0(A)/\inn(A)) = \sum_{ s, t\in Q_0} |(Q_1)_{s,t}|\cdot |(Q_{\ge1})_{s,t} | - |Q_0| + r.\]
Now $\aut(A)/ \aut_0(A)$ is a discrete group implies that \[\dim(\mathrm{Out}(A)) = \dim(\aut(A)/\inn(A)) = \dim(\aut_0(A)/\inn(A))\] and the equality follows.
\end{proof}

\begin{rem} Note that $\mathrm{Out}(A)$ is not invariant under Morita equivalence in general, while the group $\aut_0(A)/\inn(A)$ is always a Morita invariance for finite dimensional algebras \cite{po}.
\end{rem}

\begin{rem}
 Compare with the dimension formula of the first Hochschild cohomology of a path algebra \cite[Proposition 1.6]{ha}, which computes the $\dim( \operatorname{Der(A)}/ \operatorname{InnDer(A)})$, here $\operatorname{Der(A)}$ and $\operatorname{InnDer(A)}$ denote the space of derivations and inner derivations respectively. Recall that $\operatorname{Der(A)}$ and $\operatorname{InnDer(A)}$ are the Lie algebras of $\aut(A)$ and $\inn(A)$ respectively.
\end{rem}

By \cite[Corollary 1.6]{ha}, $\dim(\operatorname{Der(A)}/ \operatorname{InnDer(A)}) = 0$ if and only if $Q$ is a tree, that is, the underlying graph of $Q$ contains no cycles. Moreover, $\aut_0(A)/\inn(A)$ is connected since $\aut_0(A)$ is.  Combining with Proposition \ref{prop-identity-cpt-aut-coalg}, we have

\begin{cor}\label{cor-autogp-eq-inngp} Let $Q$ be a finite acyclic quiver. Then $\aut_0(kQ^a) = \inn(kQ^a)$ if and only if $Q$ is a tree; if and only if $\mathrm{Out}(A)$ is a finite group.
\end{cor}

Now we turn to a more general situation. Let $B$ be an arbitrary finite dimensional elementary algebra. Elementary means that each simple $B$-module is of one dimension. As we mentioned in Section 1.3, $B\cong kQ^a/I$ for some finite quiver $Q$ and an admissible ideal $I$ of $kQ^a$. Let $D = B^*$ be the dual coalgebra of $A$, then $D$ is a subcoalgebra of $kQ^c$. By Theorem \ref{thm-large-auto-ext}, we have the following result.

\begin{thm}\label{thm-autogp-fd-alg-lifting} Let $B = kQ^a/I$, where $Q$ is a finite quiver and $I$ an admissible ideal of $kQ^a$. Then $\aut(B)$ is a subquotient group of $\aut(\widehat{kQ^a})$. In particular, if $Q$ is acyclic, then $\aut(B)$ is a subquotient group of $\aut(kQ^a)$.
\end{thm}

\begin{rem} When the quiver $Q$ is not acyclic, an automorphism of $B$ needs not to be induced from an automorphism of $kQ^a$ in general. An easy example is as follows.

Let $Q$ be the quiver with exactly one vertex and one arrow. Then $kQ^a\cong k[x]$. Let $I = \langle x^3\rangle$ and $B = k[x]/I$. Consider the automorphism $\tau$ of $B$ given by  $\tau(\bar x)= \bar x + \bar x^2$. Then $\tau $ is not obtained from any automorphism of $k[x]$. Compare with Example \ref{polynomial-alg}.
\end{rem}

\section{Examples}

In this section, we consider some typical quivers and calculate the automorphism group and outer automorphism group of their path algebras.

\subsection{Quivers of directed $A_n$ type}
\begin{exm}\label{exm-a-type} Let $Q $ be the quiver of directed $A_n$ type, pictorially $Q$ is of the shape
 \[\overset{1}\bullet\longrightarrow\overset{2}\bullet\longrightarrow\overset{3}
  \bullet\longrightarrow\cdots\longrightarrow\overset{n}\bullet,\] and $A = kQ^a$ the path algebra. We claim that $\aut(A)\cong T_n/k^\stimes$, where \[T_n=\left\{{\scriptscriptstyle \begin{pmatrix}
                                            k_1 & *   & *      & * \\
                                                & k_2 & *      & * \\
                                                &     & \ddots & * \\
                                                &     &        & k_n  \\
                                          \end{pmatrix}},
 k_i\in k^\stimes\right \}\subseteq GL_{n,k}\] is the group of invertible $n\times n$ upper triangular matrices with entries in $k$, and $k^\stimes$ is identified with the subgroup of $T_n$ consisting of multiples of identity. Clearly $T_n/k^\stimes$ has a section $\widetilde{T}_n$ \[\left\{{\scriptscriptstyle \begin{pmatrix}
                                            1   & *   & *      & * \\
                                                & k_2 & *      & * \\
                                                &     & \ddots & * \\
                                                &     &        & k_n  \\
                                          \end{pmatrix}},
 k_i\in k^\stimes, i\ge 2\right \}\] in $T_n$, and this froms a subgroup which is isomorphic to $T_n/k^\stimes$.
\end{exm}

The claim is easy to prove. In fact, $A$ is isomorphic to the algebra of upper triangular $n\times n$ matrices with entries in $k$ and $A^\stimes = T_n$. Moreover, $Q$ is a tree with $\aut(Q) = \{1\}$ implies that $\aut(A) = \inn(A) = A^{\stimes}/Z(A^\stimes)$ and the conclusion follows.

Next we consider large subcoalgebras of $C=kQ^c$ and their Galois groups. Recall that for each $1\le i\le j\le n$, there exists an unique path starting at $e_i$ and terminating at $e_j$, which we denote by $E_{i,j}$.

Let $D$ be a large subcoalgebra of $C$. For each $1\le i\le n$, let $r_i$ denote the maximal integer such that $E_{i, r_i}\in D$. Clearly $r_{n-1} = r_n = n$, $i+1\le r_i\le n$ for each $1\le i\le n-1$ and $r_i \ge r_{i-1}$ for each $2\le i\le n$. One shows easily that $D$ is uniquely determined by the vector $(r_1, r_2, \cdots, r_{n})\in \mathbb{N}^n$. By direct calculation we show that
\[\Gal(C/D)=\{(X_{i,j})\in \widetilde{T}_n\mid X_{i,i} = 1, X_{i, l}=0,  1\le i\le n, i+1\le l\le r_i\}.\]
Moreover, one shows easily $\aut(C;D)=\aut(C)$, i.e., $\sigma(D)=D$ for each $\sigma\in \aut(C)$, and hence $\aut(D)\cong\aut(C)/\Gal(C/D)$ follows.

\subsection{Generalized Kronecker quivers}
\begin{exm} Let $Q = K_n$ be the generalized Kronecker quiver
\[ \xymatrix{ 1\bullet \ar@<+2.0ex>[rr]|{\alpha_1} \ar@<+1.1ex>[rr] \ar@<-2.0ex>[rr]^{ \vdots}|{\alpha_n} & &\bullet2 }\] with $n$ arrows and $A$ the path algebra. Thus $\aut(A) \cong GL(V)\ltimes V$, where $V= \oplus_{1\le j\le n}k\alpha_j =kQ_1$ and $GL(V)$ acts on $V$ canonically.

In fact, in this case $\aut(Q) = \{1\}$ again and hence $\aut(A) = \aut_0(A)$. For any $ v \in V$, we denote by $t_v\in \aut(A)$ the automorphism with $t_v(e_1)=e_1+ v, t_v(e_2) = e_2-v$ and $t_v(\alpha_j)=\alpha_j$ for $1\le j\le n$. For each $X\in GL(V)$, we denote by $\sigma_X\in \aut(A)$ the automorphism with $\sigma_X(e_i) = e_i, i=1, 2$ and $\sigma_X(\alpha_i) = X\cdot \alpha_j$ for any $1\le j\le n$. Now we have $\aut_\bullet(A) = \{\sigma_X\mid X\in GL(V)\}\cong GL(V)$ and $\inn_\circ(A) = \{ t_v\mid v\in V \}\cong V$, and the action of $\aut_\bullet(A)$ on $\inn_\circ(A)$ is given by the canonical action of $GL(V)$ on $V$. More explicitly, $\aut(A)$ is given by the matrix group
\[
                \left\{\begin{pmatrix}
                              GL_n(k) & *   \\
                              0       & 1  \\
                    \end{pmatrix}
   \right\}\subseteq GL_{n+1}(k).
\]
\end{exm}

It is direct to show that $\inn_0(A)=\{\sigma_{a\cdot\id_V}\mid a\in k^\stimes\}$. Thus by Theorem \ref{thm-antiiso-dual-alg}(5), we have $\mathrm{Pic}(A)\cong \mathrm{Out}(A)\cong PGL(V)$, the projective general linear group of $V$.

\subsection{$n$-subspace quivers}
\begin{exm} Let $Q$ be the $n$-subspace quiver, that is a star-like quiver of the shape
\[ \xymatrix{ 1\bullet \ar[rd]_{\alpha_1} & \cdots &\bullet n \ar[ld]^{\alpha_n} \\
& \bullet 0  &
},\]
 and $A$ its path algebra, here $n\ge 1$ is a positive integer. Then $\aut(Q)\cong \mathfrak{S}_n$, the symmetric group of $\{1,2,\cdots, n\}$ and
\[
  \aut_0(A) =\left\{\begin{pmatrix}
                           k_1 &        &     & \lambda_1 \\
                               & \ddots &     & \vdots \\
                               &        & k_n & \lambda_n \\
                               &        &     & 1 \\
                    \end{pmatrix}, \lambda_i\in k,  k_i\in k^\stimes
   \right\}
\]
a subgroup of $GL_{n+1, k}$. Now $\aut(A) = \mathfrak{S}_n \ltimes \aut_0(A)$ with the action of $\sigma\in \mathfrak{S}_n$
on $\aut_0(A)$ given by
\[\sigma\cdot\begin{pmatrix}
                    k_1 &        &     & \lambda_1 \\
                        & \ddots &     & \vdots \\
                        &        & k_n & \lambda_n \\
                        &        &     & 1 \\
              \end{pmatrix} =
              \begin{pmatrix}
                    k_{\sigma^{-1}_1} &        &                   & \lambda_{\sigma^{-1}_1}\\
                                      & \ddots &                   & \vdots \\
                                      &        & k_{\sigma^{-1}_n} & \lambda_{\sigma^{-1}_n} \\
                                      &        &                   & 1 \\
              \end{pmatrix}.
\]

By comparing $\Omega_\bullet(Q)$ and $\inn_0(Q)$, one shows easily that $\aut_\bullet(A)=\inn_0(A)$. Again by Theorem \ref{thm-antiiso-dual-alg}(5), we have $\mathrm{Pic(A)}\cong\mathrm{Out}(A)\cong  \mathfrak{S}_n$.
\end{exm}

\subsection{$\aut(k[x])$ and $\aut(k[[x]])$}

As we mentioned before, if the quiver contains oriented cycles, then the natual correspondence between the automorphism group of the path algebra and the one of the path coalgebra given as in Lemma \ref{lem-coalg-to-alg} may not be one to one. We will discuss this in more detail with the following example.

\begin{exm} \label{polynomial-alg}
 Let $Q$ be the quiver given by one vertex with a loop attached. Then the path algebra $kQ^a$
is isomorphic to the polynomial algebra $k[x]$ in one variable. The complete path algebra $\widehat{kQ^a}$ is given by the algebra of power series $k[[x]]$ in one variable.

It is easy to show that $\aut(k[x])\cong
\left\{\begin{pmatrix}
1& \mu\\
0& \lambda
\end{pmatrix}, \mu\in k, \lambda\in k^\stimes \right\}$, here each $\begin{pmatrix}
1& \mu\\
0& \lambda
\end{pmatrix} $ corresponds to the automorphism of $k[x]$ mapping $x$ to $\lambda x + \mu$.
\end{exm}

On the other hind, $\Omega^*(Q)=\{(\lambda_n)_{n\ge 1}, \lambda_1\in k^\stimes, \lambda_i\in k, i=2, 3, \cdots\}$, here each $(\lambda_n)_{n\ge 1}$ corresponds to the trans-datum $\lambda$ given by $\lambda_{x^n} = \lambda_n x$.

Let $\lambda = (\lambda_n)_{n\ge1}\in \Omega^*$. Let $\mathbbm{f}_\lambda \in \aut(kQ^c)$ and  $\mathbbm{f}^*_\lambda\in \aut(k[[x]])$ be the corresponding automorphisms. One shows that $\mathbbm{f}^*_\lambda(k[x]) \subseteq k[x]$ if and only if $ \lambda_n = 0$ for sufficiently large $n$; $\mathbbm{f}^*_\lambda(k[x])=k[x]$ if and only if $\lambda_n =0$ for any $n\ge 2$, and such automorphisms  of $kQ^c$ correspond to the subgroup $
\left\{\begin{pmatrix}
1& 0\\
0& \lambda
\end{pmatrix}, \lambda\in k^\stimes \right\}$ of $\aut(k[x])$.

\begin{rem} The above example shows that for a general quiver, the construction as in Lemma \ref{lem-coalg-to-alg} gives only a correspondence between certain subgroup of $\aut(kQ^a)$ and certain subgroup of $\aut(kQ^c)$. To determine this subgroup is still open for us.
\end{rem}

\noindent{\bf Acknowledgements:} The work is partly done during the author's visit at the University of Paderborn with a support by Alexander von Humboldt Stiftung. He would like to thank Professor Henning Krause and the faculty of Institut f\"{u}r Mathematik for their hospitality.  Thanks should also go to Professor Manuel Saor\'{\i}n for helpful communications, especially for pointing out the reference \cite{gs2}.

\vskip20pt
{\small\noindent  Department of Mathematics, University of Science and Technology of China, China\\
Wu Wen-Tsun Key Laboratory of Mathematics, USTC, Chinese Academy of Sciences, China\\
{\it Email}: yeyu@ustc.edu.cn}

\end{document}